\documentclass{article}%[12pt]{extarticle}
\usepackage[left=3cm,right=3cm,top=3cm]{geometry}
\usepackage{amssymb,amsthm,amsmath,amsxtra,dsfont,appendix,stix}
\usepackage[nottoc]{tocbibind}   % add bibliography to table of contents
\usepackage{hyperref} \hypersetup{colorlinks=true, citecolor=blue, linkcolor=black}
\usepackage{mathrsfs}
\usepackage[square,numbers]{natbib}
\usepackage{float}

\numberwithin{equation}{section}

\usepackage{todonotes}
\usepackage{tcolorbox}

\newcommand{\1}{\mathds{1}}
\newcommand{\A}{\mathscr{A}}

\newcommand{\Co}{\mathscr{C}}
\newcommand{\C}{\mathds{C}}
\newcommand{\D}{\mathds{D}}
\renewcommand{\d}{\mathrm{d}}
\newcommand{\E}{\mathds{E}}
\newcommand{\g}{\vec{\gamma}}
\newcommand{\G}{\mathrm{G}}
\renewcommand{\i}{\mathbf{i}}
\renewcommand{\O}{\mathcal{O}}
\renewcommand{\P}{\mathds{P}}

\newcommand{\M}{\mathrm{M}}
\newcommand{\N}{\mathbb{N}}
\newcommand{\R}{\mathds{R}}
\newcommand{\T}{\mathscr{T}}
\newcommand{\X}{\mathrm{X}}

\newcommand{\dbar}{\overline{\partial}}

\newcommand{\supp}{\operatorname{supp}}

\newtheorem{theorem}{Theorem}[section]
\newtheorem{assumption}[theorem]{Assumptions}
\newtheorem{convention}[theorem]{Conventions}
\newtheorem{proposition}[theorem]{Proposition}
\newtheorem{corollary}[theorem]{Corollary}
\newtheorem{lemma}[theorem]{Lemma}
 \newtheorem{remark}[theorem]{Remark}

\title{
\vspace*{-2cm}
Maximum of the characteristic polynomial of the Ginibre ensemble}
\date{\today}
\author{Gaultier Lambert\thanks{University of Zurich, Switzerland. E-mail: \href{mailto:gaultier.lambert@math.uzh.ch}{gaultier.lambert@math.uzh.ch} }}

\begin{document}

\maketitle

%\tableofcontents

% Keywords: Ginibre ensemble of random matrices; Maximal fluctuations of characteristic polynomial; Gaussian multiplicative chaos; Determinantal point processes; Asymptotics of Bergman kernel.

\begin{abstract}{\normalsize
We compute the leading asymptotics for the maximum  of the (centered) logarithm of the absolute value of the characteristic polynomial, denoted $\Psi_N$,  of the Ginibre ensemble as the dimension of the random matrix $N$ tends to $+\infty$.
The method relies on the log-correlated structure of the field $\Psi_N$ and we obtain the lower--bound for the maximum by constructing a family of Gaussian multiplicative chaos measures associated with certain regularization of $\Psi_N$ at small mesoscopic scales. 
We also obtain the leading  asymptotics for the dimensions of the sets of thick points and verify that they are consistent with the predictions coming from the Gaussian Free Field.
A key technical input is the approach from \cite{AHM15} to derive the necessary asymptotics, as well as the results from \cite{WW19}. 
}

%As a side result, we obtain a CLT for linear statistics 
\end{abstract}

\section{Introduction and main results}

The Ginibre ensemble is the canonical example of a non--normal random matrix. It consists of a $N\times N$ matrix filled with independent complex Gaussian random variables of variance $1/N$,  \cite{Ginibre65}. 
It is well--known that the eigenvalues $(\lambda_1, \dots , \lambda_N)$ of a Ginibre matrix are asymptotically uniformly distributed inside the unit disk $\mathbb{D} = \{ z\in \mathbb{C} : |z| < 1\}$ in the complex plane -- this is known as the circular law \cite{BZ98, BC14}.
The Ginibre eigenvalues have the same law as the particles in a one component two--dimensional Coulomb gas confined by the potential  $Q(x) = |x|^2/2$ at a specific temperature, see \cite{Serfaty18}.
That is, the joint law of the eigenvalues is given by
$\d \P_N  \propto e^{- \mathrm{H}_N(x)} {\textstyle\prod_{j=1}^N} \d^2 x_j$
where the energy of a configuration $x\in\C^N$ is
\begin{equation} \label{Hamiltonian}
 \mathrm{H}_N(x) : = \sum_{\substack{ j, k =1 , \dots , N \\ j\neq k}} \log |x_j-x_k|^{-1} + 2 N \sum_{j=1, \dots ,N} Q(x_j) , 
\end{equation}
and  $\d^2x$ denotes the Lebesgue measure on $\C$.
Moreover, these eigenvalues form a determinantal point process on $\C$ with a correlation kernel 
\begin{equation} \label{Gkernel}
K_N(x,z) =  {\textstyle\sum_{j=0}^{N-1}} \frac{x^j \overline{z}^j}{j!}  N^{j+1} e^{-N |x|^2/2 - N|z|^2/2}.
\end{equation}
This means that all the correlation functions (or marginals $\P_{N,n}$) of this point process are given by 
\begin{equation} \label{Ginibre}
\P_{N,n}[\d x_1, \cdots , \d x_n]  =  \frac{1}{(N)_n} \det\big[K_N(x_i,x_j) \big]_{i,j =1}^n  \tfrac{\d^2x_1}{\pi} \cdots \tfrac{\d^2x_n}{\pi} , 
\qquad\text{for } n=1, \dots , N,
\end{equation}
where $(N)_n = N(N-1)\cdots (N-n+1)$. 
We refer to \cite[Chapter 4]{HKPV09} for an introduction to determinantal processes and to \cite[Theorem~4.3.10]{HKPV09} for a derivation of the Ginibre correlation kernel.

In this article we are interested in the asymptotics of the modulus of the characteristic polynomial $z \in \C \mapsto \prod_{j=1}^N|z-\lambda_j|$ of the Ginibre ensemble and in particular on the maximum size of its fluctuations. Before stating our main result, we need to review some basic properties of the Ginibre eigenvalues process. 

\medskip

First, it follows from a classical result in potential theory that the \emph{equilibrium measure} which describes the limit of the empirical measure $\frac{1}{N} \sum_{j=1}^N \delta_{\lambda_j}$ is indeed the circular law: $\sigma(\d x) = \frac{1}{\pi} \1_{\D} \d^2 x$, see \cite[Section 3.2]{Serfaty18}. 
Let $(x)_+ = x \vee 0 $ for $x\in\R$. 
This can be deduced from the fact that the logarithmic potential of the circular law 
\begin{equation} \label{varphi}
\varphi(z)  : = \int \log|z-x| \sigma(\d x)
= (\log|z|)_+ - \frac{(1- |z|^2)_+}{2} ,
\end{equation}
satisfies the condition 
\begin{equation} \label{EL}
 \varphi(z) = Q(z) - 1/2  \qquad\text{for all $z\in\D$}.
 \end{equation}

\medskip
 
 Then, Rider--Vir\`ag \cite{RV07} showed that the fluctuations of the empirical measure of the Ginibre eigenvalues around the circular law are described by a Gaussian noise. This result was generalized to other ensembles of random matrices in \cite{AHM11, AHM15}, as well as to two--dimensional Coulomb gases at an arbitrary positive temperature in \cite{BBNY19,LS18}. 
Let us define
\begin{equation} \label{X}
\X(\d x) : = {\textstyle  \sum_{j=1}^N \delta_{\lambda_j}} - N \sigma(\d x) . 
\end{equation}
 This measure describes the fluctuations of the Ginibre eigenvalues and, by \cite[Theorem~1.1]{RV07}, for any function $f\in \Co^2(\C)$ with at most exponential growth, we have as $N\to+\infty$, 
 \begin{equation} \label{clt}
 \X(f) =  {\textstyle  \sum_{j=1}^N f(\lambda_j)} - N \int f(x)\sigma(\d x)   
\ \overset{\rm law}{\longrightarrow}\ \mathscr{N}\left(0, \Sigma^2(f) \right) . 
 \end{equation}
If $f$ has compact support inside the support of the equilibrium measure, then the asymptotic variance is given by 
 \begin{equation} \label{Sigma}
 \Sigma^2(f)=  \int  \dbar f (x)  \partial f(x) \sigma(\d x) . 
 \end{equation}
% where $\dbar = \frac 12 \left( \partial_x + \i \partial_y$
 
The object that we study in this article is the \emph{centered} logarithm of the Ginibre characteristic polynomial: 
\begin{equation} \label{logcharpoly}
\Psi_N(z) : = \log\left({\textstyle\prod_{j=1}^N } |z-\lambda_j| \right) - N \varphi(z) .  
\end{equation}
See Figure~\ref{fig:logcharpoly} below for a sample of the random function $\Psi_N(z)$.
Note that it follows from the convergence of the empirical measure to the circular law that for any $z\in\C$, we have in probability as $N\to+\infty$,
\[
\frac 1N \log\left({\textstyle\prod_{j=1}^N } |z-\lambda_j| \right) \to \varphi(z) ,  
\]
so that the second term on the RHS of \eqref{logcharpoly}  is necessary to have the field $\Psi_N$ asymptotically centered. 
In fact, it follows from the result of Webb--Wong \cite{WW19} that  $\E_N[\Psi_N(z)] \to 1/4$ for all $z\in\D$ as $N\to+\infty$. 
Moreover, if we interpret $\Psi_N$ as a random generalized function, then the central limit theorem \eqref{clt} implies that $\Psi_N$ converges in distribution to the Gaussian free field (GFF)\footnote{We briefly review the definition of the GFF in Section~\ref{sect:gmc}. } on~$\D$ with free boundary conditions, see \cite[Corollary 1.2]{RV07} and also \cite{AHM15, Webb} for further details.
Even though the GFF is a random distribution, it can be though of as a \emph{random surface} which corresponds to the two--dimensional analogue of Brownian motion, \cite{Sheffield07}. 
The convergence result of Rider--Vir\`ag  indicates that we can think of the field $\Psi_N$ as an approximation of the GFF in $\D$. 
The main feature of the GFF is that it is a log-correlated Gaussian process on $\C$. 
This log-correlated structure is already visible for the absolute value of the Ginibre  characteristic polynomial as it is possible to show that for any $z, x\in\D$, 
\begin{equation} \label{Psicov}
\E_N\left[ \Psi_N(z) \Psi_N(x) \right] =  \frac{1}{2} \log\left(\sqrt{N} \wedge |x-z|^{-1}\right) + \O(1) ,
\end{equation}
as $N\to+\infty$.
By analogy with the GFF and other log-correlated fields, we can make the following prediction regarding the maximum of the field  $\Psi_N$.
We have as $N\to+\infty$, 
\begin{equation} \label{max}
\max_{z\in\D} \Psi_N(z)  = \frac{\log N }{\sqrt{2}}  - \frac{3 \log \log N}{4\sqrt{2}} +\xi_N ,
\end{equation}
where the random variable $\xi_N$ is expected to converge in distribution. 
Analogous predictions have been made for other  log-correlated fields associated with \emph{normal} random matrices.
For instance, Fyodorov--Keating \cite{FK14} first conjectured the asymptotics of the maximum of the logarithm of the absolute value of the characteristic polynomial of the circular unitary ensemble\footnote{A random $N\times N$  matrix sampled from the Haar measure on the unitary group.} (CUE), including the distribution of the error term and Fyodorov--Simm \cite{FS16} made analogous prediction for the Gaussian Unitary Ensemble\footnote{A random $N\times N$ Hermitian matrix with independent Gaussian entries suitably normalized.} (GUE) . 

\medskip 

The main goal of this article is to verify the leading order in the asymptotic expansion \eqref{max}. More precisely, we prove the following result:

\begin{theorem}  \label{thm:Ginibre}
For any $0<r<1$ and any $\epsilon>0$, it holds
\[
\lim_{N\to+\infty} \P_N\left[ \frac{1-\epsilon}{\sqrt{2}} \log N \le  \max_{|z| \le r} \Psi_N(z) \le   \frac{1+\epsilon}{\sqrt{2}} \log N  \right] =1 . 
\]
\end{theorem}

It is worth pointing out that like many other asymptotic properties of the eigenvalues of random matrices, we expect  the results of  Theorem~\ref{thm:Ginibre}, as well as the prediction \eqref{max} modulo the limiting distribution of $\xi_N$, and Theorem~\ref{thm:TP} below to be \emph{universal}. This means that these results should hold for other random normal matrix ensembles with a different confining potential $Q$ as well as for other non--Hermitian Wigner ensembles under reasonable assumptions on the entries of the random matrix. 
In the remainder of this section, we review the context and most relevant results related to Theorem~\ref{thm:Ginibre}, and we provide several  motivations to study the characteristic polynomial of the Ginibre ensemble.

\subsection{Comments on Theorem~\ref{thm:Ginibre} and further results}

The study of characteristic polynomials  for different ensembles of random matrices is an interesting and active topic because of its connections to  several problems in diverse areas of mathematics. 
In particular, there are the analogy between the  logarithm of the absolute value of the characteristic polynomial of the CUE and the Riemann $\zeta$-function \cite{KS00}, as well as the connections with Toeplitz or Hankel determinant with Fisher--Hartwig symbols, e.g. \cite{Krasovsky07, DIK14, Charlier19, CG19}. 
Of essential importance is also the connection between characteristic polynomial of random matrices, log-correlated fields and the theory of Gaussian multiplicative chaos \cite{HKO01, FK14}. 
 This connection has been used in several recent works to compute the asymptotics of the maximum of the logarithm of the characteristic polynomial for various ensembles of random matrices. 
For the CUE, a result analogous to Theorem~\ref{thm:Ginibre}  was first obtained by Arguin--Belius--Bourgade \cite{ABB17}.  Then, the correction term was computed by Paquette--Zeitouni \cite{PZ18} and the counterpart of the conjecture \eqref{max} was established  for the circular $\beta$-ensembles for general $\beta>0$ by Chhaibi--Madaule--Najnudel\footnote{They obtain tightness of the appropriately centered maximum for both real and imaginary part of the logarithm of the characteristic polynomial. See also \cite{L19} for the asymptotics of the measures of thick points  by a different approach.} \cite{CMN18}. 
For the characteristic polynomial  of the GUE, as well as other Hermitian unitary invariant ensembles, the law of large numbers for the maximum of the absolute value of the characteristic polynomial was obtained in \cite{LP19}. 
Cook and Zeitouni \cite{CZ} also obtained a law of large numbers for the maximum of the characteristic polynomial of a random permutation matrix, in which case their result does not match with the prediction from Gaussian log-correlated field because of arithmetic effects.
These results rely on the log-correlated structure of characteristic polynomials  and proceed by analogy with the case of branching random walk using a modified second moment method, \cite{Kistler15}.
This method has also been successful to compute the asymptotics of the Riemann $\zeta$-function in a random 
interval of the critical line, see \cite{SW, Najnudel18, ABBRS19, Harper_b}. 
Further recent results on the deep connections between log-correlated fields, Gaussian multiplicative chaos and characteristic polynomials of $\beta$-ensembles can be found in \cite{CFLW,CN,LP}
In particular,  we prove in~\cite{CFLW} the counterpart of  Theorem~\ref{thm:Ginibre}  for the imaginary part of the characteristic polynomial of  a large class of Hermitian unitary invariant ensembles and show that this implies \emph{optimal rigidity bounds}  for the eigenvalues.
Likewise, by adapting the proof of the upper--bound in Theorem~\ref{thm:Ginibre}, we can obtain  \emph{precise rigidity estimates} for  linear statistics of the Ginibre ensemble in the spirit of \cite[Theorem 1.2]{BBNY17} and  \cite[Theorem 2]{LS18}.

\begin{theorem} \label{thm:concentration}
For any $0<r<1$ and $ \kappa >0$, define
\begin{equation} \label{classF}
 \mathscr{F}_{r,\kappa} : = \left\{  f\in \Co^2(\C) : \Delta f(z) =0\ \text{ for all } z\in \C \setminus \D_r \text{ and }\max_\C |\Delta f| \le N^\kappa  \right\} . 
\end{equation}
For any $\eta>0$ (possibly depending on $N$ with $\eta \le \frac{N}{\log N}$),  there exists a constant $C_{r}>0$ such that 
\[
\P_N\bigg[  \sup\left\{ |\X(f)| : f\in \mathscr{F}_{r,\kappa} \text{ and } \int_\D |\Delta f(z)| \frac{\d^2z}{\pi}  \le 1  \right\} \ge \eta \log N +1 \bigg] 
\le C_r N^{5/4+\kappa-\eta}.
\]
\end{theorem}

We believe that Theorem~\ref{thm:concentration} is of independent interest since it covers any smooth mesoscopic linear statistic at arbitrary small scales in a uniform way. 
This is to be compared to the local law of \cite[Theorem 2.2]{BYY14} which is valid for general Wigner ensembles, but not with the (optimal) logarithmic bound for the fluctuations and without such uniformity in~$f$.  
The proof of Theorem~\ref{thm:concentration} is given in Section~\ref{sect:concentration} and it relies on the basic observation that in the sense of distribution, the Laplacian of the field $\Psi_N$ is related to the empirical measure of the Ginibre ensemble suitably centered: 
$\Delta \Psi_N = 2\pi N \left( \frac 1N \sum_{j=1}^N \delta_{\lambda_j} - \frac 1\pi \1_\D  \right)$. 

\medskip

The proof of Theorem~\ref{thm:Ginibre} consists of an upper--bound\footnote{See Theorem~\ref{thm:WW} below.} which is based on the subharmonicity of the logarithm of the absolute value of the Ginibre characteristic polynomial and the moments asymptotics from Webb--Wong  \cite{WW19} and of a lower--bound which exploits the log-correlated structure of the field $\Psi_N$. 
More precisely, by relying on the robust approach from \cite{LOS18}, we obtain the lower--bound in Theorem~\ref{thm:Ginibre}  by constructing a family of subcritical Gaussian multiplicative chaos measures associated with certain mesoscopic regularization of the field $\Psi_N$ -- see Theorem~\ref{thm:gmc} below for further details.
Gaussian multiplicative chaos (GMC) is a theory which goes back to Kahane \cite{Kahane85} and it aims at encoding geometric features of a  log-correlated field by means of a family of random measures. 
These GMC measures are defined by taking the exponential of a log-correlated field through a renormalization procedure.
We refer the readers to Section~\ref{sect:gmc} for a brief overview of the theory and to the review of Rhodes--Vargas \cite{RV14}  or the elegant and short article of Berestycki \cite{Berestycki17}  for  more comprehensive presentations. 
It is well--known that in the subcritical phase, these GMC measures \emph{live} on the sets of so--called \emph{thick points}\footnote{
The concept of thick points is crucial to describe the geometric properties of log-correlated fields. Informally, these points corresponds to the extremal values of the field.} of the underlying  field, \cite[Section 4]{RV14}. 
 By exploiting this connection, we obtain from our analysis the leading order of the measure of the sets of thick points of the characteristic polynomial for large $N$.

\begin{theorem} \label{thm:TP}
Let us define the set of  $\beta$-thick points of the Ginibre characteristic polynomial:
\begin{equation} \label{TP}
\T_N^\beta(r): = \big\{ x\in \overline{\D_r} : \Psi_N(x)  \ge \beta \log N  \big\} 
\end{equation}
and let $|\T_N^\beta(r)|$ be its Lebesgue measure.
For any $0<r<1$, any $0 \le \beta < 1/\sqrt{2}$ and any small $\epsilon>0$, we have
\begin{equation} \label{TP0}
\lim_{N\to+\infty} \P_N\left[ N^{-2\beta^2-\delta}  \le |\T_N^\beta(r)| \le N^{-2\beta^2+\delta}\right] =1. 
\end{equation}
\end{theorem}

The proof of Theorem~\ref{thm:TP} will be given in Section~\ref{sect:TP} and the result has the following interpretation. 
By \eqref{logcharpoly}, the field $-\Psi_N$ corresponds to the  (electrostatic) potential energy  generated by the random charges $(\lambda_1, \dots , \lambda_N)$ and the negative uniform background $\sigma$.  
One may view $-\Psi_N$ as a complex energy landscape and the asymptotics \eqref{TP0} describe the multi--fractal spectrum of the level sets near the extreme local minima of this landscape. 
Moreover,  as a consequence  of Theorems~\ref{thm:Ginibre} and~\ref{thm:TP}, we obtain the leading order of the corresponding \emph{free energy}, i.e. the logarithm of the partition function of the Gibbs measure $e^{\beta \Psi_N}$ for $\beta>0$. Namely, by adapting the proof of \cite[Corollary~1.4]{ABB17}, it holds  for any $0<r<1$,  in probability, 
\vspace{-.3cm}
\begin{equation}\label{freezing}
 \lim_{N\to+\infty} \frac{1}{\beta \log N} \log\left( \int_{\D_r} e^{\beta \Psi_N(z)} \frac{\d^2z}{\pi} \right)  = \max_{\gamma \in [0, 1/\sqrt{2}]}\Big\{ \frac 1\beta+ \gamma  -\frac2\beta \gamma^2 \Big\} 
 = \begin{cases}
\displaystyle\frac1\beta+ \frac\beta8 , & \beta \in [0, \sqrt{8}] \\
\displaystyle 1/\sqrt{2} , &\beta > \sqrt{8}
 \end{cases}. 
\end{equation}

The fact that the free energy is constant and equal to ${\displaystyle\lim_{N\to+\infty}}\frac{\max_{\D_r} \Psi_N}{\log N}$ in the \emph{supercritical regime} $\beta>\sqrt{8}$ is called \emph{freezing}. 
This property is typical for Gaussian log-correlated fields and our results rigorously establish that the Ginibre characteristic polynomial behave according to the Gaussian predictions which is a well--known heuristic in random matrix theory. 
Moreover, this \emph{freezing scenario} is instrumental to predict the full asymptotic behavior \eqref{max} of the maximum of the field $\Psi_N$, including the law of the error term, see e.g. \cite{FB08}. 
For an illustration of level sets of the random function and in particular of the geometry of thick points, see Figure~\ref{fig:levelsets}.

\medskip

Let us return to the connections between our results and the theory  of Gaussian multiplicative chaos.
The family of GMC measures associated to the GFF are called \emph{Liouville measures} and they play a fundamental role in recent probabilistic constructions in the context of quantum gravity, imaginary geometries, as well as conformal field theory. We refer to the reviews \cite{Aru, RV17} for further references on these aspects of the theory. 
Thus, motivated by the result of Rider--Vir\`ag, it is expected that  a random measure whose density is given by a small\footnote{That is in the subcritical phase -- the critical value being $\gamma_*=\sqrt{8}$ as in \eqref{freezing} or  in Theorem~\ref{thm:gmc} below. }  power of the characteristic polynomial (see Figure~\ref{fig:charpoly} below)  converges when  suitably normalized:
\begin{equation} \label{charpoly}
\frac{\prod_{j=1}^N|z-\lambda_j|^\gamma}{\E_N\left[\prod_{j=1}^N|z-\lambda_j|^\gamma \right] } \frac{\d^2z}{\pi} = \frac{e^{\gamma \Psi_N(z)}}{\E_N\left[ e^{\gamma \Psi_N(z)}\right]}  \frac{\d^2z}{\pi}
\ \overset{\rm law}{\longrightarrow} \ \mu_\G^\gamma , 
\end{equation}
where $\mu_\G^\gamma$ is a Liouville measure with parameter $0<\gamma< \sqrt{8}$. 
Hence, this provides an interesting connection between the Ginibre ensemble of random matrices and random geometry. 
As we observed in \cite[Section 3]{CFLW}, this convergence result in the subcritical phase implies the lower--bound in Theorem~\ref{thm:Ginibre}. An important observation that we make in this paper is that it suffices to establish the convergence of  
$\frac{e^{\gamma \psi_N(z)}}{\E_N\left[ e^{\gamma \psi_N(z)}\right]}  \frac{\d^2z}{\pi}$ to a GMC measure for a suitable regularization $\psi_N$ of the field $\Psi_N$ in order to capture the correct leading order asymptotics of its maximum and thick points. The main issues are to work with a regularization at an \emph{optimal mesoscopic scale} $N^{-1/2+\alpha}$ for arbitrary small $\alpha>0$ and to be able to obtain the convergence in the whole subcritical phase. 
In particular, our result on GMC, Theorem~\ref{thm:gmc},  provides  strong evidence that  the prediction \eqref{charpoly} holds.

\medskip

It is an important and challenging problem to obtain \eqref{charpoly} already in the subcritical phase. In particular, this requires to derive the asymptotics of joint moments of the characteristic polynomials. For a single $z\in\D_r$, such  asymptotics are obtained by Webb--Wong in  \cite{WW19} using Riemann--Hilbert techniques. Let us recall their main result which is also a key input in our method.

\begin{theorem}[\cite{WW19}, Theorem 1.1] \label{thm:WW}
For any fixed $0<r<1$, we have
\begin{equation} \label{WW19}
\E_N[e^{\gamma \Psi_N(z)}] = (1+ o(1)) \frac{(2\pi)^{\gamma/4}}{G(1+\gamma/2)} N^{\gamma^2/8} ,
\end{equation}
where the error term is uniform for $\gamma$ in compact sets of $\{ \gamma \in \C : \Re\gamma>-2\}$ and $z\in \D_r$.
\end{theorem}

\begin{remark} \label{rk:Ginibremoment}
{\normalfont The asymptotics of the joint exponential moments of $\Psi_N$ remain conjectural, see e.g. 
\cite[Section~1.2]{WW19}, except for even moments for which there are explicit formulae, see \cite{AV03,FK07,FR09}. 
These formulae rely on  the determinantal structure of the Ginibre ensemble: for any $n\in\N$, we have for  any $z_1, \dots , z_n \in\C$ such that  $z_1 \neq \cdots \neq z_n$,
\begin{equation} \label{mcharpoly}
\E_N\left[  {\textstyle\prod_{i=1}^n \prod_{j=1}^N } |z_i-\lambda_j|^2\right]
=   \frac{ \pi^n  \prod_{k=N}^{N+n-1} k! }{N^{-Nn - \frac{n(n+1)}{2}} }  \frac{\det_{n\times n}[K_{N+n}(z_i,z_j)]}{  \prod_{1\le i< j \le n}| z_i-z_j|^2}   e^{N \sum_{i=1}^n |z_i|^2} , 
\end{equation}
where $K_{N+n}$ is the Ginibre kernel as in \eqref{Gkernel}. Using the \emph{off--diagonal} (Gaussian) decay of the Ginibre kernel,  we can show that
\[
\det_{n\times n}[ K_{N+n}(z_i,z_j)] = {\textstyle \prod_{k=1}^{n} } K_{N+n}(z_i,z_i) \left( 1+ \O(N^{-1}) \right) ,
\]
uniformly for all $\inf_{i\neq j} |z_i -z_j| \ge c\sqrt{\frac{\log N}{N}}$  if  $c>0$ is a sufficiently large constant.
If $|z_i| \le c\sqrt{\frac{\log N}{N}}$, we also have 
$K_{N+n}(z_i,z_i)  =  \tfrac N\pi \left( 1+ \O(N^{-1}) \right)$.
Thus, by \eqref{EL} and \eqref{logcharpoly}  we obtain that for any given $z_1, \dots , z_n \in \D$, such that $z_1 \neq \cdots \neq z_n$,
\begin{equation}  \label{evenmoment}
\E_N\big[ {\textstyle \prod_{i=1}^n} e^{2 \Psi_N(z_i)}\big] = 
 \left( \sqrt{2\pi N} \right)^n  |\triangle(z_1,\dots, z_n) |^{-2}   \left( 1+ \O(N^{-1}) \right) , 
\end{equation}
which  matches exactly with the \emph{Fisher--Hartwig  predictions} from \cite[Section~1.2]{WW19} with $\gamma_1= \cdots = \gamma_n =2$.}
\hfill $\blacksquare$
\end{remark}

\begin{figure}[H]
\vspace*{-.5cm}
\caption{\small \label{fig:logcharpoly}Sample of the logarithm of the absolute value of the Ginibre characteristic polynomial $\Psi_N(z)$ for $z\in\D$ for a random matrix of dimension $N=3000$.}
\centering
\includegraphics[width=.7\textwidth]{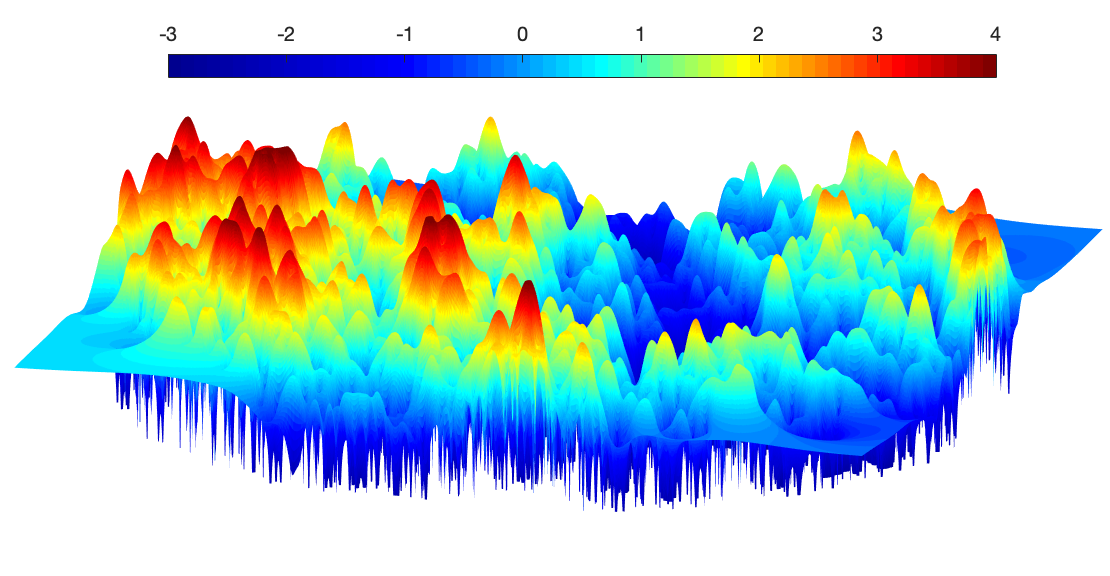}
\end{figure}

\begin{figure}[H]
\vspace*{-.5cm}
\caption{\small \label{fig:levelsets}Level sets of the logarithm of the absolute value of the Ginibre characteristic polynomial $\Psi_N(z)$ for a random matrix of dimension $N=5000$. }
\centering
\includegraphics[width=.5\textwidth]{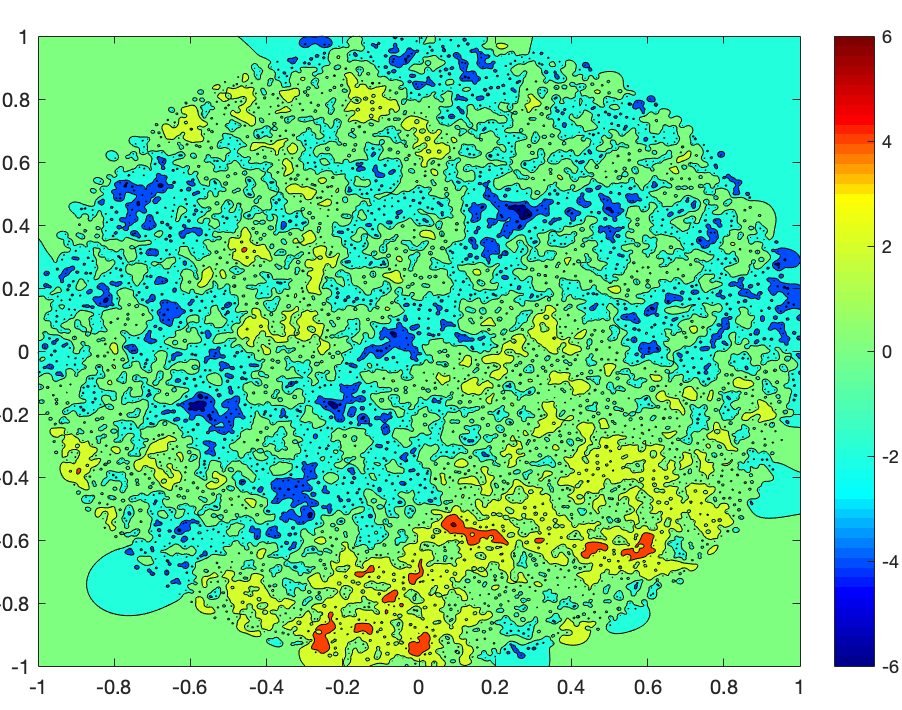}
\end{figure}

\begin{figure}[H]
%\vspace*{-1cm}
\caption{\small \label{fig:charpoly} Sample of the (normalized) Ginibre characteristic polynomial $\frac{\prod_{j=1}^N  |z-\lambda_j| }{\E_N[ \prod_{j=1}^N  |z-\lambda_j| ] }$ for a random matrix of dimension $N=3000$.  This is an approximation of the Liouville measure $\mu_G^\gamma$ with (subcritical) parameter $\gamma=1$.}
\centering
\includegraphics[width=.7\textwidth]{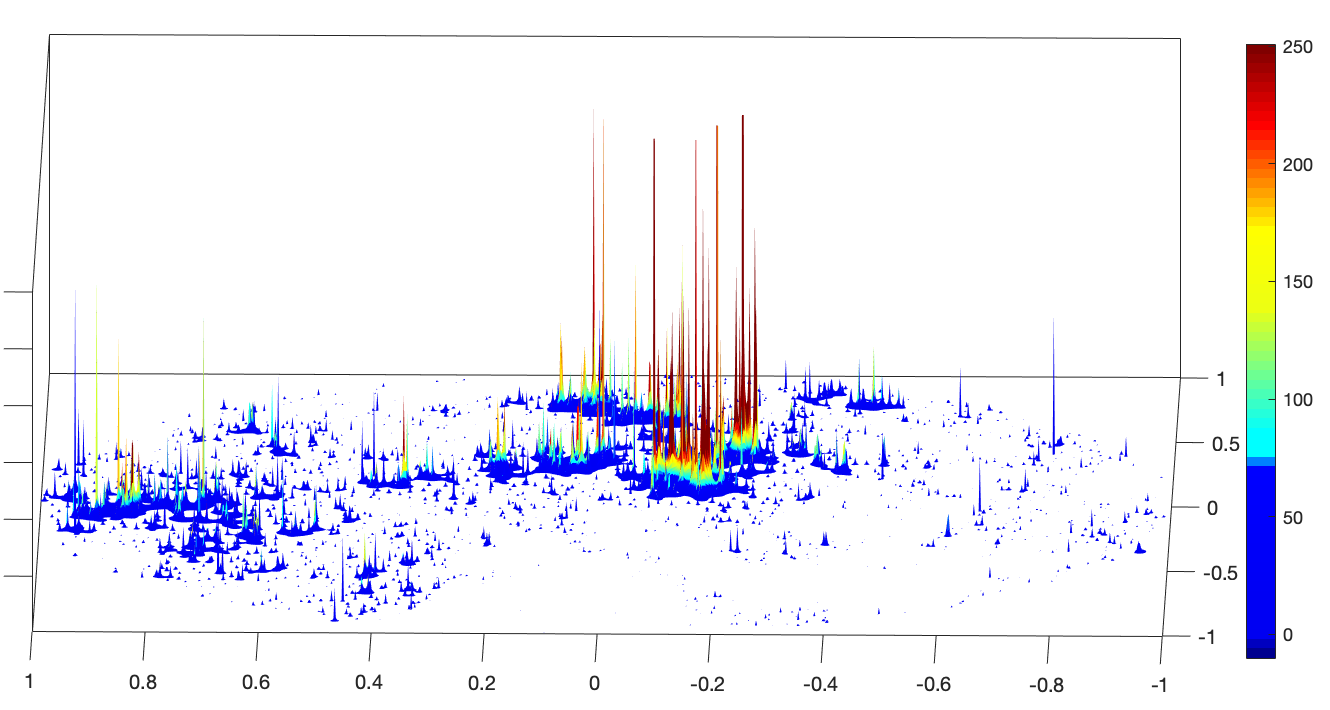}
\end{figure}

\subsection{Outline of the article}
The remainder of this article is devoted to the proof of Theorem~\ref{thm:Ginibre}. 
The result follows directly by combining the upper--bound of
Proposition~\ref{prop:UB} and the lower--bound from Proposition~\ref{prop:LB}. 
As we already emphasized the proof of the lower--bound follows from the connection with GMC theory and the details of the  argument  are reviewed in Section~\ref{sect:proofLB}. In particular, it is important to obtain \emph{Gaussian asymptotics} for the exponential moments of a mesoscopic regularization of the field $\Psi_N$, see Proposition~\ref{thm:exp}. These asymptotics are obtained by using the method developed by Ameur--Hedenmalm--Makarov \cite{AHM15} which relies on \emph{Ward identity}, also known as \emph{loop equation},  and the determinantal structure of the Ginibre ensemble. Compared with the proof of the central limit theorem in \cite{AHM15}, we face two significant extra technical challenges: we must consider a mesoscopic linear statistic coming from a test function which develops logarithmic singularities as $N\to+\infty$.  
This implies that we need a more precise approximation for the correlation kernel of the \emph{biased determinantal process}.
For these reasons, we give a detailed proof of Proposition~\ref{thm:exp} in Section~\ref{sect:clt} and Section~\ref{sect:approx}.
Our proof for the upper--bound is given in Section~\ref{sect:UB} and it relies on the subharmonicity of the logarithm of the absolute value of the Ginibre characteristic polynomial and the asymptotics from Theorem~\ref{thm:WW}. 
In Section~\ref{sect:concentration}, we discuss an application to linear statistics of the Ginibre eigenvalues and give the proof of Theorem~\ref{thm:concentration}. 

\subsection{Acknowledgment}
G.L. is  supported by the University of Zurich Forschungskredit grant FK-17-112 and  by the grant SNSF Ambizione  S-71114-05-01.
G.L. wishes to thank P. Bourgade for 
insightful discussions about the problem considered here and the referee for  interesting comments which helped to improve the presentation of this article and for pointing out several references.

\section{Proof of the lower--bound} \label{sect:proofLB}

Recall that $\Psi_N$ denotes the centered logarithm of the absolute value of the Ginibre characteristic polynomial, \eqref{logcharpoly}. The goal of this section is to obtained the following result:

\begin{proposition} \label{prop:LB}
For any $r>0$ and any $\delta>0$, we have 
\[
\lim_{N\to+\infty} \P_N\left[ \max_{|x| \le r} \Psi_N(x) \ge\frac{1-\delta}{\sqrt{2}} \log N \right] =1 . 
\]
\end{proposition}

Our strategy to prove Proposition~\ref{prop:LB} is to obtain an analogous  lower--bound for a  mesoscopic regularization of $\Psi_N$ which is also compactly supported inside $\D$. 
Note that it is also enough to consider the maximum in a disk  $\D_{\epsilon_0} = \{x\in \C : |x|\le \epsilon_0 \}$ for a small $\epsilon_0>0$. 
To construct such a regularization, let us  fix $0< \epsilon_0 \le 1/4$ and a mollifier $\phi \in \Co^\infty_c(\D_{\epsilon_0})$ which is radial\footnote{This means that $\phi(x)$ is a smooth probability density function which only depends on $|x|$ with compact support in the disk $\D_{\epsilon_0}$. Note that we can work with any such mollifier.}.
For any $0<\epsilon<1$, we denote $\phi_\epsilon(\cdot) = \phi(\cdot/\epsilon) \epsilon^{-2}$ and to approximate the logarithm of the characteristic polynomial, we consider the test function 
\begin{equation} \label{psi}  
\psi_\epsilon(z) : = \int \log |z-x|  \phi_\epsilon(x) \d^2x . 
\end{equation}
We also denote $\psi=\psi_1$. 
For technical reason, it is simpler to work with test function compactly supported inside $\D$ -- which is not the case for $\psi_\epsilon$. 
However, this can be fixed by making the following modification: for any $z \in \D_{\epsilon_0}$, we define 
\begin{equation} \label{gN}
g_N^z(x) := \psi_{\epsilon}(x-z) - \psi(x-z) , \qquad x\in \C. 
\end{equation}
It is easy\footnote{This follows from the fact that since the mollifier $\phi$ is radial and compactly supported, $\psi_\epsilon(z) = \log |z|$ for all $|z| \ge \epsilon$ and for any $\epsilon>0$.}
to see that the function $g_N^z$ is smooth and compactly supported inside  $\D(z, \epsilon_0)$.
Since we are  interested in the regime where $\epsilon(N)\to 0$ as $N\to+\infty$, we emphasize that $g_N^z$ depends on the dimension $N\in\N$ of the matrix.   Then, the random field $z\mapsto \X(g_N^z)$ is related to the logarithm of the Ginibre characteristic polynomial as follows:
\begin{equation} \label{X4}
 \X(g_N^z)
=   \int \Psi_N(x) \phi_\epsilon(z+x) \d^2x - \int \Psi_N(x) \phi(z+x) \d^2x  .
\end{equation}

In particular, $z\mapsto  \X(g_N^z)$ is still an approximate log-correlated field. Indeed, according to \eqref{clt}, \eqref{Sigma} and formula \eqref{cov1} below, we expect that as $N\to+\infty$
\[
\E_N\left[  \X(g_N^z)  \X(g_N^x) \right] = \frac{1}{2} \log\left( \epsilon(N)^{-1}\wedge |x-z|^{-1}\right) + \O(1) . 
\]
This should be compared with formula \eqref{Psicov}.

\subsection{Gaussian multiplicative chaos} \label{sect:gmc}

Let $\G$ be the Gaussian free field (GFF) on $\D$ with free boundary conditions. That is,  $\G$ is a Gaussian process taking values in the space of  Schwartz distributions with covariance kernel:
\begin{equation} \label{cov3}
\E\left[\G(x) \G(z)\right] = \frac{1}{2} \log|z-x|^{-1} .
\end{equation}
Up to a factor of $1/\pi$, the RHS of \eqref{cov3} is the Green's function\footnote{We chose this unusual normalization in order to match with formula  \eqref{Psicov}.} for the Laplace operator $-\Delta$ on $\C$.
Because of the singularity of the kernel \eqref{cov3} on the diagonal, $\G$  is called a \emph{log-correlated field} and it cannot be defined pointwise. 
In general, $\G$ is interpreted as a random distribution valued in a Sobolev space $H^{-\alpha}(\D)$ for any $\alpha>0$, \cite{Aru}. In particular, for any mollifier $\phi$ as above and any $\epsilon>0$, we  view
\begin{equation} \label{reg}
\G_\epsilon(z) := \int \G(x) \phi_\epsilon(z+x) \d^2x
\end{equation}
as a regularization of $\G$. 

\medskip

The theory of  Gaussian multiplicative chaos  aims at defining the exponential of a log-correlated field. Since such a field is merely a  random  distribution, this is a non trivial problem. However, in the so-called \emph{subcritical phase}, this can be done by a quite simple renormalization procedure. Namely, for $\gamma>0$, we define 
$\mu^\gamma_\G = : \hspace{-.1cm}e^{\gamma \G} \hspace{-.15cm}:$ as 
\begin{equation} \label{muG}
\mu^\gamma_\G(\d x) : = \lim_{\epsilon\to0} \frac{e^{\gamma\G_\epsilon(x)}}{\E[e^{\gamma\G_\epsilon(x)}]} \sigma(\d x) . 
\end{equation}
It turns out that this limit exists almost surely as a random measure on $\D$ and that it does not depend on the mollifier $\phi$ within a reasonable class.
Moreover, in the case of the GFF normalized as in \eqref{cov3}, it is a non trivial measure if and only if the parameter $0<\gamma < \sqrt{8}$ -- this is called the \emph{subcritical phase},
\cite{RV10, Berestycki17, Aru}. 
For general log-correlated field, the theory of GMC goes back to the work of Kahane \cite{Kahane85} and in the case of the GFF, the construction $\mu^\gamma_\G$ was re-discovered by Duplantier--Sheffield \cite{DS11} and Rhodes--Vargas \cite{RV11} from different perspectives. 
In a sense, the random measure  $\mu^\gamma_\G$ encodes the \emph{geometry} of the GFF. For instance, the support of $\mu^\gamma_\G$ is a fractal set which is closely related to the concept of \emph{thick points}, \cite{HMP10}. We will not discuss these issues here and refer instead to \cite{Aru, CFLW}  for further details. Let us just point out that the  relationship between Theorem~\ref{thm:gmc} and Corollary~\ref{cor:LB} below is based on such arguments.  

\medskip

For  log-correlated fields which are only asymptotically Gaussian, especially those coming from random matrix theory such as the logarithm of the Ginibre characteristic polynomial $\Psi_N$, the theory of Gaussian multiplicative chaos has been developed in \cite{Webb15, LOS18}.  
The construction in \cite{LOS18} is inspired from the approach of Berestycki \cite{Berestycki17} and it has been recently applied to unitary random matrices in  \cite{NSW}, as well as to Hermitian unitary invariant random matrices in \cite{BWW18, CFLW}.
In this paper, we construct subcritical GMC measures coming from the regularization $ \X(g_N^z)$, \eqref{X4}, of the logarithm of the Ginibre characteristic polynomial at a scale  $\epsilon = N^{-1/2+\alpha}$ for any small $\alpha>0$.
This mesoscopic regularization makes it simpler to compute the leading asymptotics of the exponential moments
of the field $ \X(g_N^z)$ -- see Proposition~\ref{thm:exp} below.
Then, using the main results from \cite{LOS18},  it allows us to prove that the limit of the renormalized exponential $\mu^\gamma_N = : \hspace{-.1cm}e^{\gamma \X(g_N^z)} \hspace{-.15cm}:$ exists for all $\gamma>0$ in the subcritical phase and that it is absolutely continuous with respect to the GMC measure $\mu^\gamma_\G$.

\begin{theorem} \label{thm:gmc}
Recall that $0< \epsilon_0 \le 1/4$ is fixed.
Let $\gamma >0$ and $g_N^{z}$  be as in \eqref{gN} with  $\epsilon=\epsilon(N) = N^{-1/2+\alpha}$ for  a fixed $0<\alpha<1/2$. Let us define the random measure $\mu^\gamma_N$ on $ \D_{\epsilon_0}$ by
\[
\mu^\gamma_N(\d z) := \frac{\exp\left(\gamma \X(g_N^{z}) \right)}{\E_N[\exp\left(\gamma  \X(g_N^{z}) \right)]}  \sigma(\d z) . 
\]
For any $0<\gamma < \gamma_* = \sqrt{8}$, the measure $\mu^\gamma_N$ converges in law as $N\to+\infty$ with respect to the weak topology toward a random measure 
$\mu^\gamma_\infty$ which has the same law, up to a deterministic constant, as $e^{\gamma\G_1(x)}\mu^\gamma_\G(\d x)$, where $\G_1$ is the smooth Gaussian process obtained from $\G$ as in \eqref{reg} with $\epsilon=1$ and $\mu^{\gamma}_\G$ is the GMC measure \eqref{muG}. In particular, our convergence covers the whole subcritical phase. 
%\[ \frac{e^{\gamma\G_1(x)}}{Z^\gamma}  \mu^\gamma_\G(\d x) , \qquad
% Z_\gamma= \pi \E[e^{\gamma\G_1(x)}] \exp\left(\]
\end{theorem}

The proof of Theorem~\ref{thm:gmc} follows from applying \cite[Theorem 2.6]{LOS18}. Let us check that the correct assumptions hold.  First, we can deduce  \cite[Assumption 2.1, Assumption 2.2]{LOS18} from the CLT of  Rider--Vir\`ag \eqref{clt}. Indeed, for any fixed $\epsilon>0$, as  $\psi_\epsilon$ is a smooth function,  the process   $(z, \epsilon)\mapsto \X\big(\psi_\epsilon(\cdot-z)\big)$ converges in the sense of finite dimensional distributions to a mean--zero Gaussian process whose covariance is given by \eqref{Sigma}\footnote{This formula for the limiting covariance in the Rider--Vir\`ag CLT  holds for test functions which are harmonic outside of $\D$, \cite{RV07}. In particular, it can be applied to \eqref{psi} if $z\in\D$ and $\epsilon>0$ is small enough.  
Then, one deduces the counterpart of \eqref{cov2}--\eqref{cov1} holds for the field \eqref{X4}  which is supported in $\D_{2\epsilon_0}$ by linearity.}. Namely, letting $\Sigma(\cdot;\cdot)$ be the quadratic form associated with $\Sigma(\cdot)$, we have for any  $z_1, z_2 \in \D_{\epsilon_0}$ and $0<\epsilon_1, \epsilon_2 \le \epsilon_0$, 
\begin{align}
\lim_{N\to+\infty} \E_N\left[ \X\left( \psi_{\epsilon_1}(\cdot - z_1 )  \right)  \X\left(\psi_{\epsilon_2}(\cdot - z_2) \right)\right] 
&\notag 
=\Sigma\left( \psi_{\epsilon_1}(\cdot - z_1 )  ;  \psi_{\epsilon_2}(\cdot - z_2) \right)\\
&\notag
 = - \frac{1}{2} \iint  \log |x_1-x_2| \phi_{\epsilon_1}(x_1 - z_1)  \phi_{\epsilon_2}(x_2-z_2)  \d^2 x_1 \d^2 x_2 \\
 &\label{cov2}
 =  \E\left[ \G_{\epsilon_1}(z_1) \G_{\epsilon_2}(z_2) \right] \\
&\notag
 = - \frac{1}{2} \iint  \log |z_1-z_2 + \epsilon_1 u_1- \epsilon_2u_2|  \phi(u_1) \phi(u_2) \d^2 u_1 \d^2 u_2  \\
&\label{cov1}
 =  \frac{1}{2} \log\left(  |z_1-z_2 |^{-1} \wedge \epsilon_1^{-1} \wedge \epsilon_2^{-1}  \right) + \underset{\epsilon_1, \epsilon_2 \to0}{\O(1)} ,
\end{align}
where the error term is uniform. 
In particular, \eqref{cov2} shows that the process $(z, \epsilon)\mapsto \X\big(\psi_\epsilon(\cdot-z)\big)$ converges in the sense of finite dimensional distributions to  $(z, \epsilon)\mapsto \G_\epsilon(z)$ as in \eqref{reg}, which comes from mollifying a GFF. 
In this case, the \cite[Assumption 2.3]{LOS18} follows e.g. from \cite[Theorem 1.1]{Berestycki17}. 
So, the only important input to deduce Theorem~\ref{thm:gmc} is to verify  \cite[Assumption 2.4]{LOS18} which consists in obtaining \emph{Gaussian asymptotics} for the joint exponential moments of the field  $\X(g_N^z)$.
Namely, we need the following asymptotics: 

\begin{proposition} \label{thm:exp}
Fix $0<\alpha<1/2$, $R>0$ and let $\epsilon=\epsilon(N) = N^{-1/2+\alpha}$.
For any $n\in\N$,  $\gamma_1 , \dots \gamma_n \in \R$, $z_1, \dots, z_n \in \C$, we  denote
\begin{equation} \label{g}
g_N^{\g, \vec{z}} (x) : = {\textstyle \sum_{k=1}^n} \gamma_k  \left( \psi_{\epsilon_k}(x-z_k) - \psi(x-z_k) \right) , \qquad x\in \C, 
\end{equation}
with parameters $ \epsilon(N)\le  \epsilon_1(N) \le \cdots \le \epsilon_n(N) <1$. 
We have
\begin{equation} \label{exp}
\E_N\left[\exp\left( \X(g_N^{\g, \vec{z}} ) \right)\right] =  \exp\Big( \frac 12 \Sigma^2(g_N^{\g, \vec{z}} )  + \underset{N\to+\infty}{o(1)}\Big) ,
\end{equation}
where  $\Sigma$ is given by \eqref{Sigma} and the error term is uniform for all $\vec{z} \in \D_{\epsilon_0}^{\times n}$ and $\vec{\gamma} \in [-R,R]^n$.
\end{proposition}

The proof of Proposition~\ref{thm:exp} is the most technical part of this paper and it is postponed to Section~\ref{sect:clt}. 
It relies on adapting in a non-trivial way the arguments of Ameur--Hedenmalm--Makarov from \cite{AHM15}. 
In particular, our proofs relies heavily on the determinantal structure of the Ginibre eigenvalues and we need local asymptotics for the correlation kernel of the ensemble obtained after making a small perturbation of the Ginibre potential -- see Section~\ref{sect:notation}.
It turns out that these asymptotics are \emph{universal}  and can be derived using techniques inspired  from the works of Berman \cite{Berman09, Berman12} which have also been applied to study the fluctuations of the eigenvalues of  normal random matrices  in \cite{AHM10, AHM11, AHM15}.

\medskip

As an important consequence of Theorem~\ref{thm:gmc}, we obtain the following corollary:  

\begin{corollary} \label{cor:LB}
Fix $0<\alpha<1/2$, let $\epsilon=\epsilon(N) = N^{-1/2+\alpha}$ and let $\psi_\epsilon$ be as in \eqref{psi}. If $\gamma_*=\sqrt{8}$, then for any $\delta>0$ and any $0< \epsilon_0 \le 1/4$, we have 
\begin{equation*} 
\lim_{N\to+\infty} \P_N\left[ \max_{|z| \le \epsilon_0} \X\left(\psi_\epsilon(\cdot-z)\right) \ge (1-\delta)\frac{\gamma^*}{2} \log \epsilon^{-1} \right] =1 . 
\end{equation*}
\end{corollary}

The proof of Corollary~\ref{cor:LB} follows from \cite[Theorem 3.4]{CFLW} with   a few non trivial modifications, the details are given in Section~\ref{sect:LB}. 

\subsection{Proof of Proposition~\ref{prop:LB}.}

We are now ready to complete the proof of Proposition~\ref{prop:LB}.
%
%\begin{proof}[Proof of Proposition~\ref{prop:LB}]
%
Observe that by \eqref{logcharpoly} and \eqref{psi}, we have for $z\in\C$ and $0<\epsilon \le 1$,
\[
\X\left(\psi_\epsilon(\cdot-z)\right) = \int \Psi_N(z+x)  \phi_\epsilon(x) \d^2x .
\]
In particular since $\supp(\phi_\epsilon) \subseteq \D_{\epsilon_0}$ for any $0<\epsilon \le 1$, this implies that we have a deterministic bound for any $z\in \C$, 
\begin{equation*} 
\X\left(\psi_\epsilon(\cdot-z)\right) \le  \max_{x\in \D(z,\epsilon_0)} \Psi_N(x) . 
\end{equation*}
Then  
\[
\max_{|z| \le \epsilon_0} \X\left(\psi_\epsilon(\cdot-z)\right)  \le  \max_{|x|\le 2\epsilon_0} \Psi_N(x)
\]
and by Corollary~\ref{cor:LB} with $\alpha=\delta$, we obtain
\[
\lim_{N\to+\infty} \P_N\left[ \max_{|x| \le 2\epsilon_0} \Psi_N(x) \ge\frac{1-3\delta}{\sqrt{2}} \log N \right] =1 . 
\]
Since  $0< \epsilon_0 \le 1/4$  and $0<\delta<1/2$ are arbitrary, this yields the claim.   
%\end{proof}

\subsection{Proof of Corollary~\ref{cor:LB}} \label{sect:LB}

This corollary follows from the results on the behavior of extreme values for general log-correlated fields which are asymptotically Gaussian developed in \cite[Section 3]{CFLW}.   
Let us fix $0<\epsilon_0 \le 1/4$. 
First of all, we verify that it follows from Proposition~\ref{thm:exp} and formula \eqref{cov1} that for any $\gamma\in\R$, as $N\to+\infty$
\[
\E_N\left[\exp\left(\gamma \X(g_N^{z}) \right)\right] =  \exp\Big( \frac{\gamma^2}{4} \log \epsilon(N)^{-1}  + \O(1)\Big) ,
\]
uniformly for all $z\in\D_{\epsilon_0}$.
These asymptotics show that the field $z\mapsto \X(g_N^{z})$ satisfies \cite[Assumptions 3.1]{CFLW} on the disk $\D_{\epsilon_0}$.
Moreover, by Theorem~\ref{thm:gmc}, $\mu^\gamma_N(\D_{\epsilon_0}) \to \mu^\gamma_\infty(\D_{\epsilon_0}) $ in distribution as $N\to+\infty$ where $0<\mu^\gamma_\infty(\D_{\epsilon_0}) <+\infty $ almost surely.
This follows from the fact that  the random measure $\mu^\gamma_\infty(\d x) \propto e^{\gamma\G_1(x)}\mu^\gamma_\G(\d x)$, $\G_1$ is a smooth Gaussian process on $\D$, $\D_{\epsilon_0}$  is a continuity set for the GMC measure $ \mu^\gamma_\G$ and $0<\mu^\gamma_\G(\D_{\epsilon_0}) <+\infty $ almost surely.
Thus,  \cite[Assumptions 3.3]{CFLW} holds and we can apply\footnote{Note that our normalization does not match with the standard  convention for log-correlated fields used in \cite[Section 3]{CFLW}. Actually,    we apply \cite[Theorem 3.4]{CFLW} to the field $\X(z)= \sqrt{2} \X(g_N^z)$ -- this explains why the critical value is $\gamma^*= \sqrt{8}$ as well as the factor $\frac{\gamma^*}{2}$ in \eqref{max1}. } \cite[Theorem 3.4]{CFLW} to obtain a lower--bound for the maximum of the  field $z\mapsto \X(g_N^{z})$.
This shows that for any $0< \epsilon_0 \le 1/4$ and any $\delta>0$, 
\begin{equation} \label{max1}
\lim_{N\to+\infty} \P_N\left[ \max_{|z| \le \epsilon_0}  \X(g_N^{ z}) \ge \left(1-\frac{\delta}{2} \right) \frac{\gamma_*}{2} \log \epsilon(N)^{-1}  \right] =1 .
\end{equation}
Let us point out that heuristically, the lower--bound \eqref{max1}  follows from the facts that the random measure  $\mu^\gamma_N$
from Theorem~\ref{thm:gmc} has most of its mass in the set $\left\{  z\in \D_{\epsilon_0} :  \X(g_N^{z}) \ge \gamma(1-\delta)\Sigma^2(g_N^{z}) \right\}$ for large $N$ and that $\mu^\gamma_N$  is a non-trivial measure if and only if $\gamma< \gamma_*$. 
Moreover, by \cite[Proposition 3.8]{CFLW}, we also obtain a lower--bound for the measure of the sets where the field $z\mapsto  \X(g_N^{ z})$ takes extreme values. Namely, under the assumptions of Proposition~\ref{thm:gmc},  we have  for any $0 \le \gamma < \frac{\gamma_*}{\sqrt{2}}$ and any small $\delta>0$, 
\begin{equation} \label{TP1}
\lim_{N\to+\infty} \P_N\left[ \bigg| \Big\{ z\in\D_{\epsilon_0}  : \X(g_N^{ z}) \ge \tfrac{\gamma}{\sqrt{2}} \log \epsilon(N)^{-1} \Big\}\bigg| \ge \epsilon(N)^{(\gamma^2-\delta)/2 } \right] =1 .
\end{equation}
In Section~\ref{sect:TP}, we use these asymptotics to compute the leading order of  the measure of the sets of thick points of the Ginibre characteristic polynomial, hence proving Theorem~\ref{thm:TP}.

\medskip

Let us return to the proof of Corollary~\ref{cor:LB} and  recall that $g_N^{z} = \psi_{\epsilon}(\cdot-z) - \psi(\cdot-z)$ with $\epsilon=\epsilon(N)$. So, in order to obtain the lower--bound,  we must show  that the random variable 
$\max_{z\in \D_{\epsilon_0}} \left| \X\left(  \psi(\cdot-z) \right)\right|$ 
remains small compared to  $\log \epsilon(N)^{-1}$ for large $N\in\N$. 
To prove this claim, we rely on the following general bound.  

\begin{lemma} \label{lem:maX}
Let $ \mathscr{F}_{r,0}$ be as in \eqref{classF}. For any $0<r<1$, there exists a constant $C_r>0$ such that
\begin{equation} \label{X2}
\E_N\Big[ \big( \max_{f \in  \mathscr{F}_{r,0}} \left| \X(f) \right| \big)^2 \Big] \le C_r \left(1+\log \sqrt{N} \right) . 
\end{equation}
\end{lemma}

\begin{proof}
It follows from the estimate \eqref{ww} below that we have uniformly for all $\gamma \in [-1,1]$ and all  $z\in\D_r$,
\begin{equation} \label{LG}
\E_N\left[e^{\gamma |\Psi_N(z)|}\right]  \le \tfrac{2C_r}{\pi} N^{\gamma^2/8} . 
\end{equation}
In particular,  by Markov's inequality, this implies that for any $\lambda>0$,
\begin{equation}  \label{TB}
\P_N\left[|\Psi_N(z)| \ge \lambda\right] \le  \tfrac{2C_r}{\pi} N^{1/8} e^{- \lambda} .
\end{equation}
Observe that according to \eqref{X}, we have for any test function $f\in \Co^2(\C)$,
\begin{equation} \label{X3}
\X(f) = \frac{1}{2\pi} \int_\C \Delta f(x) \Psi_N(x) \d^2x . 
\end{equation}
In particular,  this implies that for all $f\in  \mathscr{F}_{r,0}$, 
\begin{equation*} 
| \X(f) | \le \frac{1}{2\pi} \int_{|x| \le r} |\Psi_N(x)| \d^2x . 
\end{equation*}
Then, by  Jensen's inequality,  
\begin{equation*} 
| \X(f) |^2 \le \frac{1}{4\pi} \int_{|x| \le r} |\Psi_N(x)|^2 \d^2x . 
\end{equation*}
Therefore, it holds that
\[
\E_N\Big[ \big( \max_{f \in  \mathscr{F}_{r,0}} \left| \X(f) \right| \big)^2 \Big] \le  \frac{1}{4\pi} \int_{|x| \le r}  \E_N\left[ |\Psi_N(x)|^2 \right] \d^2x
\]
Hence, to obtain the bound \eqref{X2}, it suffices to show that for all $x\in \D_r$,  
\begin{equation} \label{M2}
\E_N\left[|\Psi_N(x)|^2 \right] \le C_r( \log \sqrt{N}+ 1) . 
\end{equation}

\medskip

Let us fix $z\in \D_r$ and $\ell_N = \frac 12 \log \sqrt{N}$. Using \eqref{LG} with $\gamma= \lambda/\ell_N $, we obtain for any $0<\lambda<  \ell_N$, 
\[
\P_N[|\Psi_N(z)| \ge \lambda]  \le C_re^{ \gamma^2 \ell_N/2 - \gamma \lambda}
=  C_r e^{- \lambda^2/2\ell_N} .
\]
Then, by integrating this estimate, we obtain
\begin{equation} \label{TB1}
 \int_0^{\ell_N} \lambda \P_N\left[|\Psi_N(z)| \ge \lambda\right]  \d\lambda \le  C_r \ell_N . 
\end{equation}
Moreover, using the bound \eqref{TB}, we also have
\begin{equation} \label{TB2} \begin{aligned}
 \int_{\ell_N}^{+\infty} \lambda \P_N\left[|\Psi_N(z)| \ge \lambda\right]  d\lambda 
 &\le    C_r  N^{1/8}\int_{\ell_N}^{+\infty} \lambda  e^{-\lambda} d\lambda   = C_r N^{1/8} (\ell_N +1) e^{-\ell_N}  \\
 & \le  C_r  (\ell_N +1) , 
\end{aligned}
\end{equation}
because  $N^{1/8} e^{-\ell_N} = N^{-1/8}$. 
By  combining the estimates \eqref{TB1} and \eqref{TB2}, we obtain for any $N\in \N$,
\[\begin{aligned}
\E_N\left[|\Psi_N(z)|^2 \right] &= 2\int_0^{+\infty} \lambda \P_N\left[|\Psi_N(z)| \ge \lambda\right]  d\lambda  \le  2C_r(2\ell_N +1) . 
\end{aligned}\]
This proves the inequality \eqref{M2} and  it completes the proof. 
\end{proof}

We are now ready to complete the proof of Corollary~\ref{cor:LB}. 

\begin{proof}[Proof of Corollary~\ref{cor:LB}]
Let us recall that we let $\epsilon=\epsilon(N) = N^{-1/2+\alpha}$ for $0<\alpha<1/2$. 
Moreover, by \eqref{psi}, we have $\Delta \psi = \phi \in \Co^\infty_c(\D_{\epsilon_0})$ with $0<\epsilon_0 \le 1/4$.
Then, for any $z\in\D_{\epsilon_0}$, the function 
$x\mapsto \psi(x-z) / \| \phi\|_\infty$ belongs to $\mathscr{F}_{1/2,0}$. 
By Lemma~\ref{lem:maX} and Chebyshev's inequality, this implies that for any $\delta>0$, 
\begin{equation} \label{max2}
\P_N\left[   \max_{z\in \D_{\epsilon_0}} \left| \X\left(  \psi(\cdot-z) \right)\right| \ge \frac{\delta}{2} \log \epsilon^{-1} \right]
\le \frac{2C_{1/2} \| \phi\|_\infty^2}{ \delta^2} \frac{2+\log N }{(\log \epsilon^{-1})^2} . 
\end{equation}
In particular, the RHS of \eqref{max2} converges to 0 as $N\to+\infty$. 
Moreover, since $\X(\psi_{\epsilon}(\cdot-z)) = \X(g_N^{z}) + \X(\psi(\cdot-z))$ and $\gamma^*\ge 1$, we have 
\[\begin{aligned}
 \P_N\left[ \max_{|z| \le \epsilon_0} \X\left(\psi_\epsilon(\cdot-z)\right) \ge (1-\delta)\frac{\gamma^*}{2} \log \epsilon^{-1} \right] 
& \ge \P_N\left[ \max_{|z| \le \epsilon_0}  \X(g_N^{ z}) \ge \left(1-\frac{\delta}{2} \right) \frac{\gamma_*}{2} \log \epsilon^{-1}  \right] \\
&\qquad -\P_N\left[   \max_{z\in \D_{\epsilon_0}} \left| \X\left(  \psi(\cdot-z) \right)\right| \ge \frac{\delta}{2} \log \epsilon^{-1} \right] . 
\end{aligned}\]
By \eqref{max1} and \eqref{max2}, this implies that 
\[
\lim_{N\to+\infty}  \P_N\left[ \max_{|z| \le \epsilon_0} \X\left(\psi_\epsilon(\cdot-z)\right) \ge (1-\delta)\frac{\gamma^*}{2} \log \epsilon^{-1} \right]  =1,
\]
which completes the proof. 
\end{proof}

\section{Proof of the upper--bound} \label{sect:UB}

The goal of this section is to prove the upper--bound in Theorem~\ref{thm:Ginibre}. 
Then, in Section~\ref{sect:concentration}, we adapt the proof in order to prove Theorem~\ref{thm:concentration}.

\begin{proposition} \label{prop:UB}
For any fixed $0<r<1$ and $\varepsilon>0$,  we have
\[
\lim_{N\to+\infty} \P_N \left[ \max_{\overline{\D_r}}  \Psi_N \le   \frac{1+\varepsilon}{\sqrt{2}} \log N \right] =1 . 
\]
\end{proposition}

In order to prove Proposition~\ref{prop:UB}, we need the following consequence of Theorem~\ref{thm:WW}: for any $0<r<1$, there exists a constant $C_r>0$ such that for any $\gamma\in[-1,4]$, 
\begin{equation} \label{ww}
\E_N[e^{\gamma \Psi_N(z)}]   \le \tfrac{C_r}{\pi} N^{\gamma^2/8} . 
\end{equation}
In fact, we do not need the precise asymptotics \eqref{WW19} and the upper--bound \eqref{ww} for the Laplace transform of the field $\Psi_N$ suffices for our applications. 
For instance, it is straightforward to deduce the following  bounds.

\begin{lemma} \label{lem:TP}
Fix $0<r<1$ and recall the definition \eqref{TP} of the set $\mathscr{T}_N^\beta$ of $\beta$-thick points.
We have for any $\beta \in [0,1]$,  
\[ 
\E_N\big[|\mathscr{T}_N^\beta|\big] \le C_r N^{-2\beta^2} .
\]
\end{lemma}

\begin{proof}
By Markov's inequality, we have for any $\beta \ge 0$, 
\[ \begin{aligned}
\E_N\big[|\mathscr{T}_N^\beta|\big] & = \int_{\D_r}  \P[\Psi_N(x) \ge \beta \log N] \d^2x \\
&\le N^{-\gamma\beta} \int_{\D_r} \E\left[ e^{\gamma\Psi_N(x)}\right] \d^2x . 
\end{aligned}\] 
Taking $\gamma=4\beta$ and using the estimate \eqref{ww}, this implies the claim.
\end{proof}

For the proof of  Proposition~\ref{prop:UB}, we also need the following simple Lemma.

\begin{lemma} \label{lem:radi}
Recall that $(\lambda_1, \dots , \lambda_N)$ denotes the eigenvalues of a Ginibre random matrix. For any $\delta\in [0,1]$ $($possibly depending on $N)$, we have  for all  $N\ge 3$,
\[
\P_N\left[ \max_{j\in[N]} |\lambda_j| \ge 1+\delta\right]  \le \delta^{-1} \sqrt{N} e^{-N \delta^2/4} . 
\]
\end{lemma}

\begin{proof}
Let us recall that Kostlan's Theorem \cite{Kostlan92} states that the random variables $\left\{N |\lambda_1|^2 , \dots, N |\lambda_N|^2\right\} $ have the same law as $\left\{ \boldsymbol{\gamma}_1, \dots, \boldsymbol{\gamma}_N\right\} $
where  $\boldsymbol{\gamma}_k$ are independent random variables with distribution 
$\P[\boldsymbol{\gamma}_k \ge t] = \frac{1}{\Gamma(k)} \int_t^{+\infty} s^{k-1} e^{-s} ds$
for $k=1, \dots, N$. 
By a union bound and a change of variable, this implies that
\[ \begin{aligned}
\P_N[ \max_{j\in[N]} |\lambda_j| \ge t ]  &\le N \P[ {\boldsymbol{\gamma}_N} \ge N t] \\
&\le \frac{N^{N+1}}{\Gamma(N)} \int_{t}^{+\infty} s^{N} e^{-Ns} \frac{ds}{s} \\
&\le  \frac{N^{N+1}e^{-N}}{\Gamma(N)t} \int_{t}^{+\infty}  e^{-N\phi(s)}ds 
\end{aligned}\]
where  $\phi(s) = s- \log s -1$. Since $\phi$ is strictly convex on $[0,+\infty)$ with $\phi'(t)=1-1/t$, this implies that
\[ \begin{aligned}
\P_N[ \max_{j\in[N]} |\lambda_j| \ge t ] 
& \le\frac{N^{N+1}e^{-N- N\phi(t)}}{\Gamma(N)t} \int_0^{+\infty} e^{-N(1-\frac 1t) s} ds \\
&\le \frac{ \sqrt{N}}{\sqrt{2\pi} (t-1)} e^{-N \phi(t)} .
\end{aligned}\]
Using that $\phi(1+\delta) \ge t^2/4$ for all $\delta\in[-1,1]$, this completes the proof.
\end{proof}

We are now  ready to give  the Proof of Proposition~\ref{prop:UB}.

\subsection{Proof of Proposition~\ref{prop:UB}}

%\begin{proof}[Proof of Proposition~\ref{prop:UB}]
Fix $0<r<1$ and a small $\varepsilon>0$ such that $r' = r+2\sqrt{\varepsilon} <1$. 
%The function $\Psi_N$ is subharmonic in $\C\setminus\D$, so by the maximum principle:
%\[
%\max_\C \Psi_N = \max_{\overline{\D}} \Psi_N . 
%\]
For $z\in\C$, let $P_N(z) := {\textstyle\prod_{j=1}^N } |z-\lambda_j|$ and
recall that the logarithmic potential of the circular law is $\varphi$, \eqref{varphi}. 
  Conditionally on the event $\left\{\max_{j\in[N]} |\lambda_j| \le \frac{3}{2} \right\}$, we have the a--priori bound: $\max_{z\in\overline{\D}}|P_N(z)| \le (\frac{5}{2})^N$. 
 Since $\Psi_N = \log P_N - N\varphi $ and $-\varphi \le 1/2$, by Lemma~\ref{lem:radi},  this shows that 
\begin{equation} \label{sub8}
\P_N\left[ \max_{\overline{\D}}  \Psi_N \ge 3N \right] \le \P_N \left[\max_{j\in[N]} |\lambda_j| \ge \frac{3}{2} \right] \le 2 \sqrt{N} e^{-N/16} . 
\end{equation}
The function $\Psi_N$ is upper--semicontinuous on $\C$, so that it attains it maximum on $\overline{\D_r}$. Let $x_* \in \overline{\D_r}$ such that 
\[
\Psi_N(x_*) =  {\textstyle \max_{\overline{\D_r}}}\  \Psi_N . 
\]
Since the function $z\mapsto \log P_N(z)$ is subharmonic on $\C$, we have
for any $\delta>0$, 
\begin{equation} \label{sub1}
\Psi_N(x_*) \le \frac{1}{\pi \delta^2} \int_{\D(x_*,\delta)} \hspace{-.3cm} \log P_N(z)\ \d^2z  -  N \varphi(x_*) .
\end{equation}
Observe that as $\varphi(z)  = \frac{|z|^2-1}{2} $  for $z\in\D$, if $\D(x_*,\delta) \subset \D$, then
\[\begin{aligned}
  \frac{1}{\pi \delta^2}  \int_{\D(x_*,\delta)}   \hspace{-.3cm}\varphi(x) \d^2 x
&  = \varphi(x_*) +   \frac{1}{\pi \delta^2}  \int_{\D_\delta}   \hspace{-.3cm} u\cdot \nabla \varphi(x_*)\ \d^2 u +   \frac{1}{2\pi \delta^2}  \int_{\D_\delta}   \hspace{-.3cm} u \cdot \nabla^2 \varphi(x_*) u\ \d^2 u  \\
& =  \varphi(x_*)  +   \frac{\Delta  \varphi(x_*)}{2\pi \delta^2}  \int_{\D_\delta}   \hspace{-.3cm} |u|^2 \d^2 u  \\
&= \varphi(x_*)  + \frac{\delta^2}{2}  .
\end{aligned}\]
%where we used that \eqref{varphi}. 
By \eqref{sub1}, this implies that 
\begin{equation} \label{sub2}
\Psi_N(x_*) \le \frac{1}{\pi \delta^2} \int_{\D(x_*,\delta)} \hspace{-.3cm} \Psi_N(z) \d^2z  + \frac{N \delta^2}{2} . 
\end{equation}
Choosing $\delta=\sqrt{ \varepsilon \frac{\log N}{N}}$ in \eqref{sub2}, we obtain
\[
\Psi_N(x_*) \le  \frac{1}{\pi \delta^2}  \int_{\D(x_*,\delta)}   \hspace{-.3cm}\Psi_N(z) \d^2z + \frac{\varepsilon}{2}  \log N .
\]
On the event $\Big\{{\textstyle \max_{\overline{\D_r}}}\  \Psi_N\ge \frac{1+\varepsilon}{\sqrt{2}} \log N \Big\}$, this implies that 
\begin{equation} \label{sub6}
  \frac{1}{\pi \delta^2}  \int_{\D(x_*,\delta)}   \hspace{-.3cm}\Psi_N(z) \d^2z \ge  \Big( 
  \tfrac{1}{\sqrt{2}}+ \frac{\varepsilon}{5} \Big) \log N .
\end{equation}
On the other--hand, by \eqref{TP} with $\beta= 1/\sqrt{2}$, 
\begin{equation} \label{sub7}
 \frac{1}{\pi \delta^2}  \int_{\D(x_*,\delta)}   \hspace{-.3cm}\Psi_N(z) \d^2z \le
 \tfrac{\log N}{\sqrt{2}} +  \frac{1}{\pi \delta^2} \int_{\mathscr{T}_N^\beta(r')}\Psi_N(z) \d^2z . 
\end{equation}
Combining \eqref{sub6} and \eqref{sub7},  this implies
\[
 \int_{\mathscr{T}_N^\beta(r')}\Psi_N(z) \d^2z  \ge \frac{\varepsilon \delta^2}{2} \log N =   \frac{(\varepsilon\log N)^2}{2N} .
\]
Hence, we conclude that for any $\eta\in[0,1]$,  on the event $\Big\{  \frac{1+\varepsilon}{\sqrt{2}} \log N \le  \max_{\overline{\D_r}}  \Psi_N \le \max_{\overline{\D_{r'}}}  \Psi_N \le \frac{\varepsilon^2}{2} (\log N)^{1+\eta} \Big\}$, 
\[
|\mathscr{T}_N^\beta(r')|  \ge \frac{(\log N)^{1-\eta}}{N} .  
\]

By Lemma~\ref{lem:TP} applied with $\beta= 1/\sqrt{2}$, this implies that 
\begin{align} \notag 
\P_N\left[ \tfrac{1+\varepsilon}{\sqrt{2}} \log N \le  \max_{\overline{\D_r}}  \Psi_N \le \max_{\overline{\D_{r'}}}  \Psi_N \le \tfrac{\varepsilon^2}{2} (\log N)^{1+\eta} \right] 
&\le \P_N\left[  |\mathscr{T}_N^\beta(r')| \ge \frac{(\log N)^{1-\eta}}{N}  \right]  \\
& \label{sub3}
\le\frac{N }{(\log N)^{1-\eta}} \E_N\left[ |\mathscr{T}_N^\beta(r')| \right] \le \frac{C_{r'}}{(\log N)^{1-\eta}} . 
\end{align}
By a similar argument as \eqref{sub2}, with $\delta=  \varepsilon \sqrt{\frac{(\log N)^{1+\eta}}{2N}}$ and choosing $x_*\in \overline{\D_r}$  such that $\Psi_N(x_*) =  {\textstyle \max_{\overline{\D_{r'}}}}\  \Psi_N$, it holds conditionally on the event $ \big\{\max_{\overline{\D_{r'}}}  \Psi_N \ge  N \delta^2\big\}$, 
\[
\frac{N\delta^2}{2} =  \frac{\varepsilon^2}{4} (\log N)^{1+\eta}  \le   \frac{1}{\pi \delta^2}  \int_{\D(x_*,\delta)}   \hspace{-.3cm}\Psi_N(z) \d^2z . 
\]
Let $\mathscr{A} = \left\{ z\in \overline{\D_{r''}} : \Psi_N(z) \ge  \frac{\varepsilon^2(\log N)^{1+\eta}}{8}  \right\} $ with $r'' = r'+\epsilon <1$.   
Conditionally on the event 
$ \big\{ \tfrac{\varepsilon^2}{2} (\log N)^{1+\eta} \le \max_{\overline{\D_{r'}}}  \Psi_N$ \ 
$ \le   \max_{\overline{\D}} \Psi_N \le 3N\big\}$, this gives
\[
 \frac{3N|\mathscr{A}|}{\pi \delta^2} + \frac{\varepsilon^2(\log N)^{1+\eta}}{8}   \ge   \frac{1}{\pi \delta^2}  \int_{\D(x_*,\delta)}   \hspace{-.3cm}\Psi_N(z) \d^2z  ,
\]
so that with $\eta=1/2$, 
\begin{equation} \label{sub4}
|\mathscr{A}|  \ge   \frac{\varepsilon^4(\log N)^{3} }{16 N^2} . 
\end{equation}
A variation of the proof of Lemma~\ref{lem:TP} using the estimate \eqref{ww} with $0<r''<1$ and $\gamma=4$ shows that 
$ \E_N\left[  |\mathscr{A}| \right]  \le  C_{r''} N e^{-\varepsilon^2(\log N)^{3/2}/2}$.
By \eqref{sub4}, we conclude that
\begin{align} \notag
 \P_N \left[ \tfrac{\varepsilon^2}{2} (\log N)^{1+\epsilon} \le \max_{\overline{\D_{r'}}}  \Psi_N \le   \max_{\overline{\D}} \Psi_N \le 3N\right] 
& \le \P_N\left[ |\mathscr{A}|  \ge   \frac{\varepsilon^2(\log N)^3 }{8 N^2} \right] \\
&\notag\le 16 \varepsilon^{-4} N^2 \E_N\left[ |\mathscr{A}|  \right]  \\
&\label{sub5}
\le  16 \varepsilon^{-4} C_{r''} N^3 e^{-\varepsilon^2(\log N)^{3/2}/2}  . 
\end{align}
In order to complete the proof, it remains to observe that by combining the estimates \eqref{sub3}, \eqref{sub5} and \eqref{sub8}, we have proved that if $\varepsilon>0$ is sufficiently small, then
\[
 \lim_{N\to+\infty} \P_N\left[ \tfrac{1+\varepsilon}{\sqrt{2}} \log N \le  \max_{\overline{\D_r}}  \Psi_N \right] =0 .
\]

\subsection{Concentration for linear statistics: Proof of Theorem~\ref{thm:concentration}} \label{sect:concentration}

In order to prove Theorem~\ref{thm:concentration}, we need the following  bounds as well as Lemma~\ref{lem:radi}. 

\begin{lemma} \label{lem:concest}
Fix $\eta>0$ and $0<r<1$.  There exists a universal constant $A>0$ such that conditionally on the event $\mathscr{B} = \{ \max_{j=1,\dots, N} |\lambda_j| \le 2\} $, we have for any function $f\in \Co^2(\C)$ $($possibly depending on $N\in\N)$ which is harmonic in $\C \setminus \overline{\D_r}$, 
\begin{equation} \label{Xest1}
| \X(f) | \le 
\eta\log N  \int_\D |\Delta f(z)| \frac{\d^2z}{2\pi} 
+ C  N \sqrt{ |\mathscr{G}_\eta|} \max_\C |\Delta f| , 
\end{equation}
where $\mathscr{G}_\eta := \left\{ z\in \D_r : |\Psi_N(z)| >  \eta\log N  \right\}$ $($with $\eta>0$ possibly depending on $N\in\N)$ and $C>0$. Moreover, there exists a constant $C_r>0$ such that for any $\kappa >0$, 
\begin{equation} \label{Xest2}
\P\left[|\mathscr{G}_\eta| \ge N^{-\kappa} \right] \le C_r N^{\kappa+ 1/8-\eta } . 
\end{equation}
\end{lemma}

\begin{proof}
Observe that for any $f\in \Co^2(\C)$ which is harmonic in $\C \setminus \overline{\D_r}$,  by definition of $\mathscr{G}_\eta$, we have
\begin{equation} \label{linstatest}
\left| \int_\C\Delta f(z) \Psi_N(z)\ \d^2z \right| \le 
\eta \log N\int_{\D_r} |\Delta f(z)| \d^2z + \max_{\D_r} |\Delta f| \int_{\mathscr{G}_\eta} |\Psi_N(z)|  \d^2z .
\end{equation}
Then, by the Cauchy--Schwartz inequality, 
\[
\int_{\mathscr{G}_\eta} |\Psi_N(z)|  \d^2z \le   
\sqrt{|\mathscr{G}_\eta| \int_{\D}  |\Psi_N(z)|^2  \d^2z } 
\]
and by \eqref{logcharpoly}, it holds conditionally on the event $\mathscr{B}$, 
\[ \begin{aligned}
 \int_{\D}  |\Psi_N(z)|^2  \d^2z &\le  
2\int_{\D}  \left(\textstyle{\sum_{j=1}^N \log|z-\lambda_j|} \right)^2 \d^2z  + \frac{N^2}{2}  \int_\D (1-|z|^2)^2 \d^2z\\
&\le N\left( 2 \int_\D \textstyle{\sum_{j=1}^N \left( \log|z-\lambda_j| \right)^2}  \d^2z + \frac{8\pi}{15} N\right) \\
&\le  C^2 N^2  ,
\end{aligned}\]
where
$\displaystyle C=  \sqrt{ 2 \sup_{|x| \le 2} \int_\D \left( \log|z-x| \right)^2 \d^2z + \frac{8\pi}{15}} $
is a numerical constant. This shows that 
\[
\int_{\mathscr{G}_\eta}|\Psi_N(z)|  \d^2z \le   C N \sqrt{ |\mathscr{G}_\eta|}  .
\]
Then, according to formula \eqref{X3} and \eqref{linstatest}, 
we obtain \eqref{Xest1}. 
In order to estimate the size of the set $\mathscr{G}_\eta$, let us observe that combining \eqref{LG} with $\gamma=1$ and Markov's inequality, we obtain
\[ \begin{aligned}
\E_N\left[ |\mathscr{G}_\eta| \right]  & = \int_{|x| \le r}  \P\left[ |\Psi_N(x)| \ge \eta \log N
\right]  \d^2x \\
&\le N^{-\eta}  \int_{|x| \le r}  \E[e^{|\Psi_N(x)|}]  \d^2x \\
&\le C_r N^{1/8-\eta}  . 
\end{aligned}\]
 By Markov's inequality, this yields the estimate \eqref{Xest2}. 
\end{proof}

\begin{proof}[Proof of Theorem~\ref{thm:concentration}]

By Lemma~\ref{lem:concest}, for any test function $f \in  \mathscr{F}_{r,\kappa}$, it holds conditionally on the event $\mathscr{B} = \{ \max_{j=1,\dots, N} |\lambda_j| \le 2\} $ that for any small $\eta>0$
\[
| \X(f) | \le \eta\log N \int_\D |\Delta f(z)| \frac{\d^2z}{2\pi} 
+ C  N^{1+\kappa} \sqrt{ |\mathscr{G}_\eta|} . 
\]
Hence, this implies that if $N \in \N$ is sufficiently large, 
\[
\P_N\bigg[  \sup\left\{ |\X(f)| : f\in \mathscr{F}_{r,\kappa} \text{ and } \int_\D |\Delta f(z)| \frac{\d^2z}{\pi}  \le 1  \right\} \ge  \eta \log N +1\bigg] 
\le  \P_N\left[   |\mathscr{G}_\eta| \ge N^{-9/8-\kappa}  \right]  + \P_N\left[ \mathscr{B}^c \right] . 
\]
By Lemma~\ref{lem:radi} with $\delta=1$ and \eqref{Xest2}, we have shown that $\P_N[\mathscr{B}^c] \le \sqrt{N} e^{-N}$  and  
$  \P_N\left[   |\mathscr{G}_\eta| \ge N^{-9/8-\kappa}  \right] \le C_r N^{5/4+\kappa-\eta}$.  
By combining these estimates, this completes the proof.
\end{proof}

\section{Thick points: Proof of Theorem~\ref{thm:TP}} \label{sect:TP}

Like the proof of Theorem~\ref{thm:Ginibre}, the proof of Theorem~\ref{thm:TP}  consists of a separate upper--bound \eqref{TPUB}
 and lower--bound (Proposition~\ref{prop:TPLB} below) and it relies on similar techniques. 
In particular, the upper--bound follows directly from Lemma~\ref{lem:TP}. Namely, by Markov's inequality, we have for any $\beta\in[0,1]$ and $\delta>0$, 
\begin{equation} \label{TPUB}
\P_N\left[  |\T_N^\beta(r)| \le N^{-2\beta^2+\delta}\right] \ge 1- \frac{ C_r}{N^\delta} . 
\end{equation}
Then, to obtain the lower--bound, we rely the fact that the field $\Psi_N$ can be well approximated by 
$\X\left(\psi_\epsilon(\cdot-z)\right) $ for $\epsilon = N^{-1/2+\alpha}$ with a small scale $\alpha>0$ and use the estimate \eqref{TP1}.

\begin{proposition} \label{prop:TPLB}
For any $0<r<1$, any $0\le \beta <1/\sqrt{2}$ and any $\delta>0$, we have
\[
\lim_{N\to+\infty}  \P_N\left[ | \T_N^\beta(r) | \ge N^{-2\beta^2-\delta} \right]  =1 .
\]
\end{proposition}

\begin{proof}
We fix parameters $r\in (0,1)$, $\beta \in [0,1/\sqrt{2})$ and we abbreviate $\T_N^\beta= \T_N^\beta(r)$. 
Recall that  $\phi \in \Co^\infty_c(\D_{\epsilon_0})$ is a  mollifier and that for any $z\in\C$, 
 \begin{equation} \label{X5}
\X\left(\psi_\epsilon(\cdot-z)\right) = \int \Psi_N(x)  \phi_\epsilon(z-x) \d^2x ,
\end{equation}
where $\epsilon=\epsilon(N) = N^{-1/2+\alpha}$ -- the scale $0<\alpha<1/2$ will be chosen later in the proof depending on $\beta$ and $\delta$. 
Throughout the proof, we assume that $\epsilon$ is small compared to $\epsilon_0\le 1/4$, we let $\mathrm{c} = \sup_{x\in\C} \phi(x)$ and for a small $\delta\in (0, 1/2]$, 
\[
\Upsilon_N^\beta : = \Big\{ z\in\D_{\epsilon_0}  : \X\big(\psi_\epsilon(\cdot-z)\big) \ge (\beta+ \tfrac\delta8) \log N \Big\} .
\]
We also define the event (of large probability):
\[
\A : = \Big\{ \max_{|x| \le r} \Psi_N (x) \le  \log N\Big\} . 
\]
%Recall that $\gamma_* = \sqrt{8}$ is the critical value of the GMC measures (c.f.~Theorem~\ref{thm:gmc}) and by Proposition~\ref{prop:UB}, $\P_N[\A] \to 1$ as $N\to+\infty$. 
%
Since $g_N^{z}= \psi_\epsilon(\cdot-z) -\psi(\cdot-z)$ by \eqref{gN},  we have for any $\gamma>0$, 
\[\begin{aligned}
 \P_N\left[ \bigg| \Big\{ z\in\D_{\epsilon_0}  : \X(g_N^{ z}) \ge \tfrac{\gamma + \delta}{\sqrt{2}} \log \epsilon^{-1} \Big\}\bigg| \ge \epsilon^{\gamma^2/2- 3\delta/4}  \right] 
 \le   \P_N\left[   \max_{z\in \D_{\epsilon_0}} \left| \X\left(  \psi(\cdot-z) \right)\right| \ge \tfrac{\delta}{2\sqrt{2}} \log \epsilon^{-1} \right] \ \  \\
+ 
 \P_N\left[ \bigg| \Big\{ z\in\D_{\epsilon_0}  : \X\big(\psi_\epsilon(\cdot-z)\big) \ge \tfrac{\gamma+\delta/2}{\sqrt{2}} \log \epsilon^{-1} \Big\}\bigg| \ge \epsilon^{\gamma^2/2-3\delta/4} \right] . 
  \end{aligned}\]
Then, using the estimates \eqref{TP1} and \eqref{max2}, we obtain that for any $0 \le \gamma < \frac{\gamma_*}{\sqrt{2}}$, % (with $\delta$ sufficiently small), 
\begin{equation*}
\lim_{N\to+\infty}  \P_N\left[ \bigg| \Big\{ z\in\D_{\epsilon_0}  : \X\big(\psi_\epsilon(\cdot-z)\big) \ge \tfrac{\gamma+\delta/2}{\sqrt{2}} \log \epsilon^{-1} \Big\}\bigg| \ge \epsilon^{\gamma^2/2-3\delta/4} \right]  =1 .
\end{equation*}
Hence, choosing the scale $\alpha = \frac{\delta}{8\sqrt{2}(\gamma+\delta/2)}$  with $\gamma = \sqrt{8}\beta$, this implies that for any $0\le \beta <1/\sqrt{2}$, 
\begin{equation} \label{TP3}
\lim_{N\to+\infty}  \P_N\left[ | \Upsilon_N^\beta | \ge N^{-2\beta^2-\delta/2} \right]  =1 .
\end{equation}

\medskip

By  formula \eqref{X5} and the definition of $\beta$-thick points,  we have conditionally on $\A$, for any $z\in \D_{\epsilon_0}$, 
 \begin{equation} \label{TP4}
 \begin{aligned}
\X\left(\psi_\epsilon(\cdot-z)\right)  &= \int_{\D_r\setminus\T_N^\beta} \Psi_N(x)  \phi_\epsilon(x-z) \d^2x   +  \int_{\T_N^\beta} \Psi_N(x)  \phi_\epsilon(x-z) \d^2x \\
&\le  \beta \log N + \mathrm{c}  \big|  \T_N^\beta \cap \D(z, \tfrac\epsilon4) \big| \epsilon^{-2} \log N , 
\end{aligned}
\end{equation}
where we used that $ \phi_\epsilon(x-z) \le \mathrm{c} \epsilon^{-2} \1_{|x-z| \le \epsilon/4}$ at the last step.
Now, let us tile the disk $\D_{\epsilon_0}$ with squares of area~$\epsilon^{2}$. To be specific, let $M = \lceil \epsilon^{-1} \rceil$ and 
$\square_{i,j} = [i\epsilon, (i+1)\epsilon] \times [j\epsilon, (j+1)\epsilon]$ for all integers $i,j \in [-M,M]$. 
Note that since $z\mapsto \X\left(\psi_\epsilon(\cdot-z)\right) $ is a continuous process, for any $i,j \in \mathbb{Z}  \cap [-M,M]$, we can choose
\[
z_{i,j} = \arg\max \big\{ \X\left(\psi_\epsilon(\cdot-z)\right) :  z\in \square_{i,j}  \big\} . 
\]
The point of this construction is that we have the deterministic bound
 \begin{equation} \label{TP5}
|\Upsilon_N^\beta| \le\  \epsilon^2  \hspace{-.5cm} \sum_{i,j \in \mathbb{Z}  \cap [-M,M] } \hspace{-.3cm}  \1_{z_{i,j} \in \Upsilon_N^\beta} .
\end{equation}
Moreover if $z_{i,j} \in \Upsilon_N^\beta$,  \eqref{TP4} shows that  conditionally on $\A$, 
\[
 \big|  \T_N^\beta \cap \D(z_{i,j}, \frac\epsilon4) \big| \ge  \frac{\delta \epsilon^2}{8\mathrm{c} }  .
\]
By \eqref{TP5}, this implies that 
\[
|\Upsilon_N^\beta| \le  \frac{ 8 \mathrm{c}}{\delta}  \hspace{-.1cm}  \sum_{i,j \in \mathbb{Z}  \cap [-M,M] } \hspace{-.3cm}  \1_{z_{i,j} \in \Upsilon_N^\beta}  \big|  \T_N^\beta \cap \D(z_{i,j}, \frac\epsilon4) \big| . 
\] 
Since the squares $\square_{i,j}$ are disjoint (except for their sides) and $z_{i,j} \in \square_{i,j}$, we further have the deterministic bound
\[
  \sum_{i,j \in \mathbb{Z}  \cap [-M,M] } \hspace{-.3cm}    \big|  \T_N^\beta \cap \D(z_{i,j}, \frac\epsilon4) \big|   \le 4 |\T_N^\beta| . 
\]
Hence, we conclude that conditionally on $\A$, for $0\le \beta <1/\sqrt{2}$ and $\delta>0$ sufficiently small (but independent of $N$), 
\[
|\T_N^\beta|  \ge \frac{\delta}{32\mathrm{c} }   |\Upsilon_N^\beta| . 
\]
Finally, according to Proposition~\ref{prop:UB},  we have $\P_N[\A] \to 1$ as $N\to+\infty$, so that 
by combining the previous estimate with \eqref{TP3},  this completes the proof.
\end{proof}

\section{Gaussian approximation} \label{sect:clt}

In this section, we turn to the proof of our main asymptotic result:  Proposition~\ref{thm:exp}. 
Its proof relies on the so-called \emph{Ward's identity} or \emph{loop equation} which have already been used in \cite{AHM15} as well as \cite{BBNY17, BBNY19} to study  the fluctuations of linear statistics of eigenvalues of random normal matrices and two--dimensional Coulomb gases respectively. 
For completeness, we provide a detailed proof of  the loop equation that we use in Section~\ref{sect:Ward}. 
Then, to show that the error terms in this equation are small, we rely on the determinantal structure of the  ensemble obtained after making a small perturbation of the potential $Q$ and on a local approximation of its correlation kernel (see Proposition~\ref{prop:approx} below).
This approximation is justified in Section~\ref{sect:approx} based on the method from \cite{AHM15} and we use it to prove that the error terms are indeed negligible as $N\to+\infty$ in Sections \ref{sect:est}--\ref{sect:error3}. 
Finally, we finish the proof of Proposition~\ref{thm:exp} in Section~\ref{sect:proof} by using a classical argument introduced by Johansson \cite{Johansson98} to prove a CLT for linear statistics of $\beta$-ensembles on $\R$. 
Before starting our analysis, we need to introduce further notations. 

\subsection{Notation} \label{sect:notation}

For any $N\in\N$, we let
\begin{equation} \label{polynom}
\mathscr{P}_N = \{\text{analytic polynomials of degree}<N \} .
\end{equation}
Let us recall that by Cauchy's formula, if $f$ is smooth and compactly supported inside $\D$, we have
\begin{equation} \label{cauchy}
f(z) =  \int  \frac{\dbar f(x) }{z-x}\sigma(\d x) , 
\end{equation}
where $\sigma(\d x) = \frac{1}{\pi} \1_{\D} \d^2 x$ denotes the circular law.
Throughout Section~\ref{sect:clt}, we fix $n\in\N$, $\vec{\gamma} \in [-R,R]^n$, $\vec{z} \in \D_{\epsilon_0}^{\times n}$ and we let $g_N = g_N^{\g, \vec{z}}$ be as in formula \eqref{g}.
We recall that as $\vec{z} \in  \D_{\epsilon_0}^{\times n}$ varies, the functions $x\mapsto g_N^{\g, \vec{z}}(x)$ remain smooth and compactly supported inside $\D_{2\epsilon_0}$ for all $N\in\N$. 
Let us define for $t>0$, 
\begin{equation} \label{Pbias}
\d \P_N^* : =\frac{e^{t\X(g_N)}}{\E_N [e^{t\X(g_N)}]} \d \P_N .
\end{equation}
The biased measure $\P_N^*$ corresponds to an ensemble of the type \eqref{Hamiltonian} with a \emph{perturbed potential} $Q^* := Q- \frac{tg_N}{2N}$. Therefore, under  $\P_N^*$, $\lambda=(\lambda_1, \dots \lambda_N)$ also forms a determinantal point process on $\C$ with a correlation kernel:
\begin{equation} \label{kernel1}
k^*_N(x,z) :=   {\textstyle \sum_{k=0}^{N-1}} p_k^*(x) \overline{p_k^*(z)} ,
\end{equation}
where $(p_0^*, \dots , p_{N-1}^*)$ is an orthonormal basis  of $\mathscr{P}_N$ with respect to the inner product inherited from $L^2(e^{-2N Q^*})$ such that $\operatorname{deg}(p_k^*) = k$ for $k=0,\dots, N-1$. 
 We denote 
\begin{equation} \label{kernel2}
K_N^*(x,z) : = k_N^*(x,z) e^{-N Q^*(x) - N Q^*(z)} 
\end{equation} 
and we define the \emph{perturbed one--point function}: $u_N^*(x) : = K_N^*(x,x) \ge 0 $. 
By definitions, we record that for any $N\in\N$ and all $x\in \C$,
\begin{equation} \label{mass}
\int_{\C} k^*_N(x,z) \d^2z = u_N^*(x)
\qquad\text{and}\qquad
\int_{\C} u_N^*(z) \d^2z =N.
\end{equation}
Finally, we set $\widetilde{u}_N^*  : = u_N^* - \sigma$, so that for any smooth function $f:\C\to\C$, we have  
\begin{equation} \label{1pt}
\E_N^*[\X(f)] =  \int f(x) \widetilde{u}_N^*(x) \d^2x .
\end{equation}

\begin{convention}  \label{conv}
As in proposition~\ref{thm:exp}, we fix a scale $0<\alpha<1/2$ and let $\epsilon=\epsilon(N) = N^{-1/2+\alpha}$. We also  fix $\beta>1$ and let
$\delta= \delta(N) = \sqrt{(\log N)^\beta /N}$ as in Proposition~\ref{prop:approx} below.
%Finally, we set $\eta= \eta(N) = \delta/\epsilon = (\log N)^{\beta/2} N^{-\alpha}$.
Throughout Section~\ref{sect:clt}, we assume that the dimension $N\in\N$ is sufficiently large so that $  \delta/\epsilon  \le 1/4$ and $( \delta/\epsilon )^\ell \le N^{-1}$ for a fixed $\ell\in\N$ -- e.g. we can pick $\ell = \lfloor 2/\alpha \rfloor$.  
Moreover, $C, N_0 >0$ are positive constant which may change from line to line and depend only on  the mollifier $\phi$, the parameters $R,\alpha,\beta, \epsilon_0>0$, $n, \ell\in\N$ and $t\in[0,1]$ above. Then, we  write $A_N=\O(B_N)$ if there exists such a constant $C>0$ such that $0\le A_N \le C B_N$.  
\end{convention}

\subsection{Ward's identity} \label{sect:Ward}

Formula \eqref{ward} below is usually called \emph{Ward's equation} or \emph{loop equation} and the terms $\mathfrak{T}^k_N$ for $k=1,2,3$ should be treated as corrections because of the factor $1/N$ in front of them. 
This equation is the key input of a method pioneered by Johansson \cite{Johansson98} to establish that linear statistics of $\beta$-ensembles are  asymptotically Gaussian. In the following, we follow the approach of Ameur--Hedenmalm--Makarov \cite[Section 2]{AHM15} who applied Johansson's method to study the fluctuations of the eigenvalues of random normal matrices, including the Ginibre ensemble.

\begin{proposition} \label{prop:ward}
If $g\in \Co^2_c(\D)$, we have  for any $N\in\N$ and $t\in(0,1]$,
\begin{equation} \label{ward}
\E_N^*\left[ \X(g) \right] =
 \Sigma(g;g_N)  +\frac{1}{N} \left(\mathfrak{T}^1_N(g) + \mathfrak{T}^2_N(g) - \mathfrak{T}^3_N(g)\right) , 
\end{equation}
where $\Sigma(\cdot;\cdot)$ denotes the quadratic form associated with \eqref{Sigma}, 
\[ \begin{aligned}
\mathfrak{T}^1_N(g) := \int \Big( t \dbar g (x)  \partial g_N(x) +\frac 14 \Delta g(x) \Big) \widetilde{u}_N^*(x)  \d x   , \qquad
\mathfrak{T}^2_N(g) :=  \iint \frac{\dbar g(x) }{x-z} \widetilde{u}_N^*(z) \widetilde{u}_N^*(x) \d^2 x\d^2z
\end{aligned}\]
and 
\[
\mathfrak{T}^3_N(g) :=  \iint \frac{\dbar g(x) }{x-z}  | K_N^*(x,z) |^2 \d^2 x\d^2z .
\]
\end{proposition}

\begin{proof}
An integration by parts gives for any $h\in \Co^1(\C)$ with compact support: 
\begin{equation} \label{ward1}
\E_N^*\left[ \sum_{j\neq k} \frac{h(x_j)}{x_j-x_k} +  \sum_{j} \partial h(x_j) - \sum_{j} h(x_j) \left(  2N \partial Q - t\partial g_N\right)(x_j)  \right] =0 .
\end{equation}
Observe that with $h = \dbar g$, by \eqref{1pt} and \eqref{cauchy}, it holds
\begin{equation} \label{X1}
\E_N^*\left[ \X(g) \right] = \int  g(z)\ \widetilde{u}_N^*(z) \d^2 z  =   \iint  \frac{h(x) }{z-x}\sigma(\d x)\ \widetilde{u}_N^*(z) \d^2z . 
\end{equation}
On the one--hand, using the determinantal formula for the second correlation function of the ensemble $\P_N^*$, we have
\begin{equation} \label{ward2}
\E_N^*\left[ \sum_{j\neq k} \frac{h(x_j)}{x_j-x_k}  \right]
= \iint \frac{h(x) }{x-z} u_N^*(x) u_N^*(z)  \d^2z\d^2x - \frac{1}{2}  \iint \frac{h(x) - h(z)}{x-z} \left| K_N^*(x,z) \right|^2  \d^2z\d^2x, 
\end{equation}
where the second term is equal to $\mathfrak{T}^3_N(g)$ and  the first term on the RHS satisfies
\begin{equation} \label{ward3}
\begin{aligned}
\frac 1N \iint \frac{h(x) }{x-z} u_N^*(x) u_N^*(z) \d^2z\d^2x
 & = N \iint \frac{h(x) }{x-z}  \sigma(\d z)\sigma(\d x) +  \iint \frac{h(x) }{x-z}  \sigma(\d z) \widetilde{u}_N^*(x) \d^2x \\
 &\qquad+ \iint \frac{h(x) }{x-z}  \sigma(\d x) \widetilde{u}_N^*(z) \d^2z
 + \frac 1N \mathfrak{T}^2_N(g) . 
  \end{aligned}
\end{equation}
On the other--hand by \eqref{EL}, $\partial Q (x) = \partial \varphi(x) = \frac 12 \int \frac{1}{x-z} \sigma(\d z) $ for all $x\in\D$ and as= $\supp(h) \subset \D$, we also have
\begin{align} \notag
\E_N^*\left[ \sum_{j} h(x_j) \partial Q (x_j)  \right] & =  N  \int h(x) \partial Q (x) \sigma(\d x) + \int h(x) \partial Q (x)\ \widetilde{u}_N^*(x) \d x \\
&  \label{ward4}
= \frac{N}{2}   \iint \frac{h(x) }{x-z}  \sigma(\d z)\sigma(\d x) 
 +\frac 12 \iint \frac{h(x) }{x-z}  \sigma(\d z)\ \widetilde{u}_N^*(x) \d x . 
\end{align}
Combining formulae \eqref{ward2}, \eqref{ward3} and \eqref{ward4}, we obtain
\[ \begin{aligned}
\E_N^*\left[ \frac 1N \sum_{j\neq k} \frac{h(x_j)}{x_j-x_k} - 2  \sum_{j} h(x_j) \partial Q (x_j)   \right] =  \iint \frac{h(x) }{x-z}  \sigma(\d x) \widetilde{u}_N^*(z) \d^2z 
 + \frac 1N \left( \mathfrak{T}^2_N(g)  - \mathfrak{T}^3_N(g) \right) . 
\end{aligned} \]
By formula \eqref{X1}, this implies that
\begin{equation} \label{ward5}
\E_N^*\left[ \frac 1N \sum_{j\neq k} \frac{h(x_j)}{x_j-x_k} - 2  \sum_{j} h(x_j) \partial Q (x_j)   \right]  
 =  - \E_N^*\left[ \X(g) \right]   +  \frac 1N \left( \mathfrak{T}^2_N(g)  - \mathfrak{T}^3_N(g) \right). 
 \end{equation}
Combining formulae \eqref{ward1} and \eqref{ward5} with $h = \dbar f$, this shows that
\begin{equation} \label{ward6}
\E_N^*\left[ \X(g) \right]  =  \frac 1N \E_N^*\left[ \sum_j  \Big(t \dbar g(x_j)  \partial g_N(x_j) +\frac 14 \Delta g(x_j) \Big) \right]
 +  \frac 1N \left( \mathfrak{T}^2_N(g)  - \mathfrak{T}^3_N(g) \right) ,
\end{equation} 
where we used that $\partial \dbar g = \frac 14 \Delta g$.  Finally using that $g\in\Co^2_c(\D)$,
$\displaystyle \int \Delta g(x) \sigma(\d x) =0 $ and by \eqref{Sigma}, we conclude that
\begin{equation} \label{ward7}
\frac 1N \E_N^*\left[ \sum_j \Big(\dbar g(x_j)  \partial g_N(x_j) +\frac 14 \Delta g(x_j) \Big)\right] 
=   t\Sigma(g;g_N)  + \frac 1N \int \Big(t\dbar g(x)  \partial g_N(x) +\frac 14 \Delta g(x) \Big) \widetilde{u}_N^*(x)  \d x   . 
\end{equation}
Combining formulae \eqref{ward6} and \eqref{ward7}, this  completes the proof.
\end{proof}

\subsection{Kernel approximation} \label{sect:kernel}

Recall that the probability measure $\P_N^*$ induces a determinantal process on $\C^N$ with correlation kernel $K_N^*$, \eqref{kernel2}.
In order to control the RHS of \eqref{ward}, we need the asymptotics of the this kernel as the dimension $N\to+\infty$. 
In general, this is a challenging problem, however it is expected that  $K_N^*$ decays quickly off diagonal and its asymptotics near the diagonal are universal in the sense that they are similar to that of the Ginibre correlation kernel $K_N$. In Section~\ref{sect:approx}, using the method from Ameur--Hedenmalm--Makarov \cite{AHM10,AHM15} which relies on H\"ormander's inequality and the properties of reproducing kernels,  we compute the asymptotics of $K_N^*$ near the diagonal. 
Recall that $g_N = g_N^{\g, \vec{z}}$ as in \eqref{g} and  our  Conventions~\ref{conv}. 
Let us also define the \emph{approximate Bergman kernel}:
\begin{equation} \label{kapprox} 
k^\#_N(x,w) := \frac{N}{\pi} e^{N x \overline{w}}  e^{-t \Upsilon_N^w(x-w) } ,
\qquad x,w\in\C. 
\end{equation}
where $\Upsilon_N^w(u) : = \sum_{i=0}^{\ell} \frac{u^i}{i!} \partial^i g_N(w)$. 
 We also let
\begin{equation} \label{K0}
K^\#_N(x,w) := k^\#_N(x,w)  e^{-NQ^*(z) - N Q^*(w)} ,
\qquad x,w\in\C. 
\end{equation}

Let us state our main approximation result for the \emph{perturbed} kernels which corresponds to \cite[Lemma A.1]{AHM11} in the case where the test function $g_N$ depends on $N\in\N$ and develops logarithmic singularities as $N\to+\infty$. Because of these  significant differences,  we adapt the proof in Section~\ref{sect:BA}.

\begin{proposition} \label{prop:approx}
Let $\vartheta_N  : = \1_\D + \sum_{k=1}^n \epsilon_k^{-2} \1_{\D(z_k, \epsilon_k)}$ and  $\delta= \delta(N) : = \sqrt{(\log N)^\beta /N}$ for $\beta>1$. There exist constants $L,  N_0>0$   such that for all $N\ge N_0$, we have for any $z\in \D_{1-2\delta}$ and all $w\in\D(z,\delta)$, 
\[
\left|K^*_N(w,z) - K^\#_N(w,z) \right| \le   L  \vartheta_N(z) 
\]
\end{proposition}

\begin{remark} \label{rk:uniformity}
{\normalfont
We emphasize again that the constants $L,  N_0>0$ do not depend on $\vec{\gamma}\in[-R,R]^n$, $\vec{z} \in \D_{\epsilon_0}^{\times n}$, nor $t\in[0,1]$. Consequently, the estimates of Sections~\ref{sect:est}--\ref{sect:error3}  bear the same uniformity even though this will not be emphasized to lighten the presentation.
In fact, since the parameter $t\in(0,1]$ is not relevant for our analysis, we will also assume that $t=1$ to simplify notation -- this amounts to changing the parameters $\vec{\gamma}$ to $t\vec{\gamma}$. 
} \hfill $\blacksquare$
\end{remark}

In the remainder of this section and in Section~\ref{sect:est},  we discuss some consequences of the approximation of Proposition~\ref{prop:approx}. 
Then, in Sections~\ref{sect:error1}--\ref{sect:error3}, we control the \emph{error} terms $\mathfrak{T}^k_N(g_N)$ for $k=1,2,3$ in order to complete the proof of Proposition~\ref{thm:exp} in Section~\ref{sect:proof}.

\medskip

By definitions, with $t=1$, we have for any $z\in\C$,
\[
K^\#_N(z,z) = \frac{N}{\pi} e^{N|z|^2 - g_N(z) - 2N Q^*(z)} = \frac{N}{\pi} . 
\]
Then according to \eqref{1pt} and by taking  $w=z$ in Proposition~\ref{prop:approx}, this implies that for any $z\in \D_{1-2\delta}$, 
\begin{equation} \label{1ptest}
| \widetilde{u}_N^*(z) | \le  L\vartheta_N(z) ,  
\end{equation}
where we used that the  circular density $\sigma(z) = 1/\pi$ if $z\in\D$.

\begin{lemma} \label{lem:04}
It holds as $N\to+\infty$,
\begin{equation*}
\int_\C | \widetilde{u}_N^*(x) | \d^2 x = \O(N \delta) . 
\end{equation*}
\end{lemma}

\begin{proof}
First, let us observe that since $\sigma$ is a probability measure supported on $\overline{\D}$, we have by \eqref{mass}, 
\begin{equation} \label{1pt1}
\int_{\C\setminus\D} | \widetilde{u}_N^*(x) | \d^2 x  = \int_{\C\setminus\D} u_N^*(x)  \d^2 x = N - \int_\D u_N^*(x)  \d^2 x . 
\end{equation}
 Moreover,  by \eqref{1ptest} and using that $\displaystyle\int_\D \vartheta_N(x) \d^2x = (n+1)\pi$, we also have 
\begin{equation} \label{1pt2}
\int_{|x| \le 1-2\delta} \hspace{-.3cm} | \widetilde{u}_N^*(x) | \d^2 x =\O(1) . 
\end{equation}
Since $\widetilde{u}_N^*  = u_N^* - \sigma$ and $\displaystyle \int_{|x| \le 1-2\delta} \hspace{-.3cm}\sigma(\d x) =  (1-2\delta)^2 $, the previous estimate shows that
\begin{equation} \label{1pt3}
 \int_\D u_N^*(x)  \d^2 x \ge \int_{|x| \le 1-2\delta} \hspace{-.3cm} u_N^*(x)  \d^2 x \ge N - \O(N\delta) .
\end{equation} 
Combining \eqref{1pt3} with formula \eqref{1pt1}, we conclude that as $N\to+\infty$, 
\begin{equation} \label{1pt4}
\int_{\C\setminus\D} | \widetilde{u}_N^*(x) | \d^2 x  = \O(N\delta) . 
\end{equation}
Moreover, using the uniform bound from Lemma~\ref{lem:02} below, there exists $C>0$ such that $ | \widetilde{u}_N^*(x) | \le C N$   for all $x\in \C $ which implies that 
\begin{equation} \label{1pt5}
\int_{ 1-2\delta \le |x| \le 1} \hspace{-.3cm} | \widetilde{u}_N^*(x) | \d^2 x =\O(N \delta) . 
\end{equation}
Combining the estimates  \eqref{1pt2},  \eqref{1pt4}  and  \eqref{1pt5}, this completes the proof.  
\end{proof}

\subsection{Technical estimates} \label{sect:est}

We denote the Gaussian density with variance $2/N$ by
$\Phi_N(u) : =  \tfrac{N}{\pi} e^{-N|u|^2}$. 
Since for any $x,z\in\C$, 
\begin{equation}  \label{Q}
NQ^*(z) + NQ^*(x) - N\Re\{z\overline{x}\} + g_N(x) =  \frac N2 |z-x|^2  + \frac{ g_N(x) - g_N(z)}{2} , 
\end{equation}
we deduce from formulae \eqref{kapprox}--\eqref{K0} with $t=1$ that
\begin{equation} \label{K3}
 \begin{aligned}
| K^\#_N(z,x) |^2 & = \tfrac{N}{\pi} \Phi_N(x-z)  e^{g_N(z)-g_N(x) -2 \Re\left\{ \sum_{i=1}^\ell \frac{(z-x)^i}{i!} \partial^i g_N(x) \right\}  } .
\end{aligned}
\end{equation}
We should view the last factor of \eqref{K3} as a correction. Indeed on small scales, i.e.~if $|x-z| \le \delta$, then $ e^{g_N(z)-g_N(x) -2 \Re\left\{ \sum_{i=1}^\ell \frac{(z-x)^i}{i!} \partial^i g_N(x) \right\}  } = 1+\O(\eta)$ where $\eta = \delta/\epsilon$ goes to 0 as $N\to+\infty$. 
 In particular, this implies that for $N$ is sufficiently large, it holds for all $x,z\in\C$ such that $|x-z| \le \delta$, 
\begin{equation} \label{K9}
|K_N^{\#}(x,z)|\le N . 
\end{equation}
Actually, formula \eqref{K3} shows that on microscopic scales,  $| K^\#_N(z,x) |^2$ is well approximated by the Gaussian kernel  $\Phi_N(x-z)$. 
As in \cite[Lemma 3.3]{AHM15}, we use this fact to prove the following Lemma\footnote{Note that our approximations are more precise than in \cite{AHM15}.}. 

\begin{lemma} \label{lem:05}
%Recall that we fixed  $\ell \in\N$ in such a way that $\eta^\ell \le N^{-1}$. There exists coefficients $c_1, \dots , c_\ell \ge 0$ independent of $N\in\N$ such that for any $x\in\D$,as $N\to+\infty$, 
It holds uniformly for all $x\in\D$, as $N\to+\infty$, 
\[
\int_{|x-z| \le \delta} \hspace{-.3cm}   | K_N^\#(z,x) |^2  \d^2z  = N \sigma(x) + \O\big( \vartheta_N(x) \big)
\]
where $\vartheta_N$ is as in Proposition~\ref{prop:approx}. 
\end{lemma}

\begin{proof}
Throughout this proof, let us fix $x\in\D$. 
Since $g_N$ is a smooth function, by Taylor's Theorem up to order $2\ell$, there exists a matrix $\mathrm{M} \in \R^{\ell \times \ell}$ (with positive entries) such that for all $u \in \D_{\delta}$, 
\[
g_N(x+u)-g_N(x) =  \underset{0<i+j < 2\ell}{\textstyle  \sum_{i,j=1}^{2\ell-1} }  \M_{i,j} u^i \overline{u}^j
\partial^{i} \dbar^j  g_N(x)  + \O\left( \big\{  \| \nabla^{2 \ell }  g_N\|_\infty \delta^{2\ell}\big\} \right) . 
\]
Let $ \mathrm{Y}^1_x (u): = \sum_{i=1}^{\ell-1} \frac{  \M_{i,i} }{4} |u|^{2i} \Delta^i g_N(x) $
and  $ \mathrm{A}^1_x (u): = \hspace{-.3cm} \underset{0<i+j < 2\ell, i\neq j}{\textstyle  \sum_{i,j=1}^{2\ell-1} }   \M_{i,j} u^i \overline{u}^j \partial^{i} \dbar^j  g_N(x) -2 \Re\left\{ {\textstyle \sum_{i=1}^\ell} \tfrac{u^i}{i!} \partial^i g_N(x)  \right\} $ for $u\in\C$.
Recall that by assumptions,  $\| \nabla^k g_N \|_{\infty}  \le C\epsilon^{-k}$ for all integer $k\in[1, 2 \ell]$ and $\eta= \delta/\epsilon$, so that with the previous notation:
\[ \begin{aligned}
g_N(x+u)-g_N(x) -2 \Re\left\{ {\textstyle \sum_{i=1}^\ell} \tfrac{u^i}{i!} \partial^i g_N(x)  \right\}
  =    {\textstyle  \sum_{i,j=1}^{\ell} }  \M_{i,j} u^i \overline{u}^j \
\partial^{i} \dbar^j  g_N(x) +\O(\eta^{2\ell}).
\end{aligned}\] 
Using the condition $\eta^\ell \le N^{-1}$, by \eqref{K3}, the previous expansion shows that for all $z\in \D(x,\delta)$, 
\begin{equation} \label{K2}
| K^\#_N(z,x) |^2  = \tfrac{N}{\pi} \Phi_N(x-z)  e^{\mathrm{A}^1_x (z-x)+ \mathrm{Y}^1_x (z-x)+\O(N^{-2})} . 
\end{equation}
Importantly, note that for $|u|\le \delta$, 
\begin{equation} \label{Aest}
| \mathrm{Y}^1_x(u)|  , | \mathrm{A}^1_x(u) |  =  \O(\eta^2) ,
\end{equation} 
and that  both $\Phi_N$  and $\mathrm{Y}^1_x$ are radial functions, so that it  holds for any $ k \in\N$, 
\begin{equation} \label{est18}
 \int_{|u| \le \delta} \hspace{-.3cm}  \left( \mathrm{A}^1_x(u) \right)^k   \exp\left( \mathrm{Y}^1_x(u) \right) \Phi_N(\d u)  = 0 .
\end{equation}

Hence, using \eqref{K2}--\eqref{est18}, this implies that for any $x\in\D$, 
\[\begin{aligned}
\int_{|x-z| \le \delta} \hspace{-.3cm}   | K_N^\#(z,x) |^2  \d^2z 
&=  \frac{N}{\pi}  \int_{|u| \le \delta} \hspace{-.3cm}   e^{\mathrm{A}^1_x (u)+ \mathrm{Y}^1_x(u)}\Phi_N(\d u)  + \O(N^{-1}) \\
&=  \frac{N}{\pi}  \int_{|u| \le \delta} \hspace{-.3cm}  e^{\mathrm{Y}^1_x (u)} \Phi_N(\d u)  + \O(N^{-1}) , 
\end{aligned}\]
where we used that $\Phi_N$ is a probability measure. 
Moreover, we verify by \eqref{g} and \eqref{psi} that $|\Delta^{k+1} g_N(x)| \le  C \epsilon^{-2k}  \vartheta_N(x)   $ for all integer $k\in[0, \ell]$, so we can bound $e^{\mathrm{Y}^1_x (u)}  =  1+ \O\big(|u|^2 \vartheta_N(x) \big)$ uniformly for all $|u|\le \delta$, 
Since for any integer $j\ge 0$,
\begin{equation} \label{Gmoment}
\int_{|u| \le \delta } |u|^{2j}   \Phi_N(\d u)  = N^{-j} \left( j! + \O(e^{- N\delta^2}) \right) , 
\end{equation}
we conclude that for all $x\in\D$, 
\begin{equation*} 
\int_{|x-z| \le \delta} \hspace{-.3cm}   | K_N^\#(z,x) |^2  \d^2z  =  \frac{N}{\pi}+ \O\big( \vartheta_N(x) \big)
\end{equation*}
with uniform errors. Since $\sigma(x)= 1/\pi$ for $x\in \D$, this completes the proof.
\end{proof}

We can use Proposition~\ref{prop:approx} and Lemma~\ref{lem:05} to estimate a similar integral for the correlation kernel $ K_N^*$. 
This corresponds to the counterpart of \cite[Corollary 3.4]{AHM15}. 

\begin{lemma} \label{lem:06}
It holds for any $x\in \D_{1-2\delta}$,  as $N\to+\infty$
\[
\int_{|x-w| > \delta} \hspace{-.3cm}|K^*_N(z,x) |^2 \d^2z
= \O\left(N \delta^2 \vartheta_N(x) \right) . 
\]

\end{lemma}

\begin{proof}
First of all, let us bound
\[ \begin{aligned}
\left| \int_{|x-w| \le \delta} \hspace{-.3cm}|K^*_N(z,x) |^2 \d^2z
- \int_{|x-w| \le \delta} \hspace{-.3cm}|K^\#_N(z,x) |^2 \d^2z \right| 
\le 2 \int_{|x-z| \le \delta} \hspace{-.3cm}|K^\#_N(z,x) |  \left|  K^*_N(z,x) - K^\#_N(z,x)  \right|\d^2z  \\
 + \int_{|x-z| \le \delta} \hspace{-.1cm}  \left|  K^*_N(z,x) - K^\#_N(z,x)  \right|^2\d^2z  .
\end{aligned}\]
According to Proposition~\ref{prop:approx}, it holds for any $x\in \D_{1-2\delta}$,
\begin{equation}  \label{K4}
\int_{|x-z| \le \delta} \hspace{-.1cm}  \left|  K^*_N(z,x) - K^\#_N(z,x)  \right|^2\d^2z = \O\left( \delta^2\vartheta_N(x)^2 \right) ,
\end{equation}
Similarly, using the estimate \eqref{K9},
\begin{equation}  \label{K5}
\int_{|x-z| \le \delta} \hspace{-.3cm}|K^\#_N(z,x) |  \left|  K^*_N(z,x) - K^\#_N(z,x)  \right|\d^2z \le  \pi L N \delta^2 \vartheta_N(x) . 
\end{equation}
As $\vartheta_N \le (n+1) \epsilon^{-2}  \le N$, this shows that for any $x\in \D_{1-2\delta}$, 
\[
\int_{|x-w| \le \delta} \hspace{-.3cm}|K^*_N(z,x) |^2 \d^2z
= \int_{|x-w| \le \delta} \hspace{-.3cm}|K^\#_N(z,x) |^2 \d^2z +\O\left(N \delta^2 \vartheta_N\right) . 
\]
Using the reproducing property \eqref{mass} and Lemma~\ref{lem:05}, we conclude  that for any $x\in \D_{1-2\delta}$,
\[ \begin{aligned}
\int_{|x-w| > \delta} \hspace{-.3cm}|K^*_N(z,x) |^2 \d^2z
 & =  u^*_N(x) - \int_{|x-w| \le \delta} \hspace{-.3cm}|K^\#_N(z,x) |^2 \d^2z + \O\left(N \delta^2 \vartheta_N\right)  \\
 & = \widetilde{u}^*_N(x) + \O\left(N \delta^2 \vartheta_N\right) .  
\end{aligned}\]
Using the estimate $| \widetilde{u}_N^*(x) | \le  L \vartheta_N(x)$, see   \eqref{1ptest}, this yields the claim. 
\end{proof}

Finally, we need a last Lemma which relies on the \emph{anisotropy} of the \emph{approximate Bergman kernel} $K_N^{\#}$ that we can already see from formula \eqref{K3}. 

\begin{lemma} \label{lem:07}
It holds as $N\to+\infty$, 
\[  \iint_{\substack{ |x| \le 1/2 \\  |z-x| \le \delta}} \hspace{-.1cm} \frac{\dbar g_N(x)- \dbar g_N(z)}{x-z} |K_N^{\#}(x,z)|^2  \d^2z \d^2x  =  \O(1) . 
\]
\end{lemma}

\begin{proof}
The proof if analogous to that of Lemma~\ref{lem:05}. 
Since $g_N$ is a smooth function, by Taylor Theorem up to order $2\ell\in\N$, it holds  for any $x\in \D_{1/2}$  and  $z\in\D(x,\delta)$,  
\[
  \frac{\dbar g_N(x)- \dbar g_N(z)}{x-z}  =  \underset{0<i+j\le 2\ell}{\textstyle  \sum_{i,j=0}^{2\ell} } \M_{i,j} u^{i-1} \overline{u}^j
\partial^{i} \dbar^{j+1}  g_N(x) \\
 + \O\left( \big\{  \| \nabla^{2(\ell+1)}  g_N\|_\infty \delta^{2\ell}\big\} \right) ,
\]
where $u = (z-x) \neq 0$.
Let $ \mathrm{Y}^2_x (u): = \sum_{j=0}^{\ell-1} \frac{  \M_{j+1,j} }{4} |u|^{2j} \Delta^{j+1} g_N(x) $
and  $ \mathrm{A}^2_x (u): = \hspace{-.3cm} \underset{0<i+j\le 2\ell, i\neq j+1}{\textstyle  \sum_{i,j=0}^{2\ell} }  \hspace{-.3cm} \M_{i,j} u^{i-1} \overline{u}^j \partial^{i+1} \dbar^{j+1}  g_N(x) $ for $u\in\C$. 
Since $\| \nabla^{2(\ell+1)} g_N\|_\infty \delta^{2\ell} \le C \eta^{2\ell} \epsilon^{-2} \le CN^{-1}$  because we choose $\ell\in\N$ in such a way $\eta^\ell \le N^{-1}$ with $\eta = \delta/\epsilon$,  
this shows that uniformly for all $x\in \D_{1/2}$  and  $z\in\D(x,\delta)$,  
\[
\frac{\dbar g_N(x)- \dbar g_N(z)}{x-z}  
= \mathrm{Y}^2_x(x-z) + \mathrm{A}^2_x(x-z)
+ \O(N^{-1}) . 
\]
By Lemma~\ref{lem:05}, we immediately see  that    $\displaystyle \iint_{\substack{ |x| \le 1/2 \\  |z-x| \le \delta}} \hspace{-.1cm}|K_N^{\#}(x,z)|^2  \d^2z \d^2x  = \frac N4+ \O(1)$ and the previous expansion implies that
\[  
\mathfrak{Z}_N : = \iint_{\substack{ |x| \le 1/2 \\  |z-x| \le \delta}} \hspace{-.1cm}\frac{\dbar g(x)- \dbar g(z)}{x-z} |K_N^{\#}(x,z)|^2  \d^2z \d^2x  
=\iint_{\substack{ |x| \le 1/2 \\  |z-x| \le \delta}} \hspace{-.1cm}\left( \mathrm{Y}^2_x(x-z) + \mathrm{A}^2_x(x-z) \right) |K_N^{\#}(x,z)|^2  \d^2z \d^2x  
+\O(1) .
\]
Using formula \eqref{K2}, \eqref{Aest} and the estimates  $|\mathrm{Y}^1_x(u)| , |\mathrm{A}^1_x(u)| = \O(\epsilon^{-2})$ which are uniform for $x,u\in\C$, we obtain  by a change of variable, 
\[  
\mathfrak{Z}_N =  \frac{N}{\pi}  \iint_{\substack{ |x| \le 1/2 \\  |u| \le \delta}}
\left( \mathrm{Y}^2_x(u) + \mathrm{A}^2_x(u) \right)e^{\mathrm{A}^1_x (u)+ \mathrm{Y}^1_x (u)}\Phi_N(u)  \d^2u \d^2x  + \O(\epsilon^{-2}N^{-1}) . 
\]
The error term will be negligible.
If we proceed exactly as in the proof of Lemma~\ref{lem:05}, see \eqref{est18}, then only the radial parts contributes:
\[  \begin{aligned} 
\mathfrak{Z}_N 
= \frac{N}{\pi}  \iint_{\substack{ |x| \le 1/2 \\  |u| \le \delta}} 
 \mathrm{Y}^2_x(u) \exp\left( \mathrm{Y}^1_x(u)\right)  \Phi_N(\d u)  \d^2u \d^2x  + \O(\epsilon^{-2}N^{-1}) . 
% &= \frac{N}{\pi}  \iint_{\substack{ |x| \le 1/2 \\  |u| \le \delta}} 
 %\left({\textstyle \sum_{j=0}^{\ell-1} }   \tfrac{c'_j}{j!}\Delta^{j+1}g_N(x) |u|^{2j}  \right)  \Phi_N(\d u)  \d^2x + \O(1) , 
\end{aligned}\]
Moreover, using that $|\Delta^{k+1} g_N(x)| \le  C \epsilon^{-2k}  \vartheta_N(x)   $ for all integer $k\in[0, \ell]$,
we can develop for all $|u| \le \delta$,  $ \mathrm{Y}^2_x(u) \exp\left( \mathrm{Y}^1_x(u)\right) = \Delta g_N(x)+ \O(\vartheta_N(x)|u|^2 )$ uniformly for all $x\in\D_{1/2}$ -- here we used again that the parameter $\eta\le 1/4$ to control the error term. Hence, by \eqref{Gmoment}, we conclude that
\[  
\mathfrak{Z}_N =  \frac{N}{\pi}   \int_{|x| \le 1/2} \hspace{-.1cm}\Delta g_N(x) \d^2x 
+ \O\bigg( \int_{|x| \le 1/2} \hspace{-.1cm} \vartheta_N(x) \d^2x   \bigg) + \O(\epsilon^{-2}N^{-1}) .
\]
Since the first integral on the RHS vanishes and the second integral is $O(1)$, this completes the proof.
\end{proof}

\subsection{Error of type $\mathfrak{T}^1_N$} \label{sect:error1}

In Sections~\ref{sect:error1}--\ref{sect:error3}, we use the estimates from Sections~\ref{sect:kernel} and~\ref{sect:est}  to bound the error terms when we apply Proposition~\ref{prop:ward} to the function $g_N = g_N^{\g, \vec{z}}$  given by \eqref{g}. Let us abbreviate
\begin{equation} \label{Sigmag}
\Sigma = \Sigma(g_N) = \sqrt{ \int  \dbar g_N(x)  \partial g_N(x) \sigma(\d x)} . 
\end{equation}

\begin{proposition} \label{prop:error1}
We have 
$\left| \mathfrak{T}^1_N(g_N) \right| = \O\left( \Sigma^2  \epsilon^{-2} \right)$,
uniformly for all $t\in(0,1]$,
as $N\to+\infty$. 
\end{proposition}

\begin{proof}
A trivial consequence of the estimate \eqref{1ptest} is that 
$|\widetilde{u}_N^*(x)| \le C \epsilon^{-2}$ for all $|x| \le 1/2$. 
Since $\supp(g_N) \subseteq \D_{1/2}$, this implies that
\[
\left| \int \Delta g_N(x)\ \widetilde{u}_N^*(x)\d^2x \right| \le C  \epsilon^{-2}  \int |\Delta g_N(x)| \d^2x  = \O( \epsilon^{-2})  , 
\]
where we used that $\Delta g_N(x) = {\textstyle \sum_{k=1}^n} \gamma_k  \left( \phi_{\epsilon_k}(x-z_k) - \phi(x-z_k) \right)$ so that 
$\displaystyle \int |\Delta g_N(x)| \d^2x \le 2  {\textstyle \sum_{k=1}^n} |\gamma_k|$ since $\phi$ is a probability density function on $\C$. Similarly, we have
\[
\left|\int \dbar g_N(x)  \partial g_N(x) \ \widetilde{u}_N^*(x)\d^2x\right| \le  C  \epsilon^{-2}  \int    \dbar g_N(x)  \partial g_N(x) \d^2x = \O\left( \Sigma^2  \epsilon^{-2} \right) , 
\]
since $\dbar g_N= \overline{\partial g_N} $ so that $\dbar g_N(x)  \partial g_N(x) \ge 0$ for all $x\in\C$ and the previous integral is equal to  $\pi \Sigma^2$.
By definition of $ \mathfrak{T}^1_N$ -- see Proposition~\ref{prop:ward} --   this proves the claim.
\end{proof}

\subsection{Error of type $\mathfrak{T}^2_N$} \label{sect:error2}

\begin{proposition} \label{prop:error2}
Recall that $\eta = \delta/\epsilon$. 
It holds as $N\to+\infty$,
$| \mathfrak{T}^2_N(g_N)|  =  \O\left( \Sigma  N \eta \right)$. 
\end{proposition}

\begin{proof}

Fix a small parameter $0< \kappa \le 1/4$  independent of $N\in\N$ and let us split
\begin{equation}  \label{error2}
 \mathfrak{T}^2_N(g_N)= \iint \frac{\dbar g_N(x) }{x-z} \widetilde{u}_N^*(z) \widetilde{u}_N^*(x) \d^2 x\d^2z
= \mathfrak{Z}_N +  \iint_{|z-x| \ge \kappa} \frac{\dbar g_N(x) }{x-z}  \widetilde{u}_N^*(z) \widetilde{u}_N^*(x) \d^2 x\d^2z 
\end{equation}
where 
\begin{equation}  \label{error1}
\mathfrak{Z}_N
: = \iint_{|z-x| \le \kappa} \frac{\dbar g_N(x) }{x-z}  \widetilde{u}_N^*(z) \widetilde{u}_N^*(x)\d^2 x\d^2z . 
\end{equation}
Since $\supp(g_N) \subseteq \D_{1/2}$, by Lemma~\ref{lem:04},  
the second term on the RHS of \eqref{error2} satisfies
\begin{equation}  \label{error5}
\left| \iint_{|z-x| \ge \kappa} \frac{\dbar g_N(x) }{x-z}  \widetilde{u}_N^*(z) \widetilde{u}_N^*(x) \d^2 x\d^2z \right|  = \O\left( N \delta   \int_{|x|\le 1/2} |\dbar g_N(x)  \widetilde{u}_N^*(x) | \d^2x \right) . 
\end{equation}
Moreover, using Cauchy--Schwartz inequality  and \eqref{1ptest}, this implies that
\[
 \int_{|x|\le 1/2}  |\dbar g_N(x)  \widetilde{u}_N^*(x) | \d^2x \le L \sqrt{ \int  |\dbar g_N(x) |^2\d^2x   \int_{|x|\le 1/2} \vartheta_N^2(x) \d^2x  } . 
\]
According to the notation of Proposition~\ref{prop:approx}, we verify ${ \displaystyle \int_{|x|\le 1/2} } \vartheta^2_N(x) \d^2x  \le \tfrac{\pi}{2}+ 2\pi \sum_{j,k=1}^n \epsilon_k^{-2} \epsilon_j^{-2}(\epsilon_k^2+\epsilon_j^{2}) \le C\epsilon^{-2}$, so that by \eqref{Sigmag}, 
\begin{equation}  \label{error6}
 \int_{|x|\le 1/2}  |\dbar g_N(x)  \widetilde{u}_N^*(x) | \d^2x = \O(\Sigma \epsilon^{-1}). 
\end{equation}
The estimates \eqref{error5} and \eqref{error6} show that with $\eta = \delta/\epsilon$, 
\begin{equation}  \label{error4}
\left| \iint_{|z-x| \ge \kappa} \frac{\dbar g_N(x) }{x-z}  \widetilde{u}_N^*(z) \widetilde{u}_N^*(x) \d^2 x\d^2z \right|  =  \O(N \Sigma \eta) . 
\end{equation}

\medskip

Let $\mathscr{S}_N :=\bigcup_{k=1}^n \D(z_k,\epsilon_k)$. 
In order to control the integral \eqref{error1},  we split it into $n+1$ parts and use  \eqref{1ptest} which is valid for all $x\in\supp(g_N)$, then we obtain 
\begin{equation} \label{error7}
 \begin{aligned}
\left| \mathfrak{Z}_N \right|
& \le L  \left( \sum_{k=1}^n \epsilon_k^{-2}  \int_{|x-z_k| \le \epsilon_k}
+ \int_{x \notin \mathscr{S}_N} \right)
\left( \int_{|w| \le \kappa}  |\widetilde{u}_N^*(x+w)|  \frac{\d^2 w}{|w|} \right)  |\dbar g_N(x)|  \d^2x   . 
\end{aligned} 
\end{equation}
On the one hand, it follows from  \eqref{1ptest} that for any $x\in\D(z_k,\epsilon_k)$, 
\[ \begin{aligned}
\int_{|w| \le \kappa}  |\widetilde{u}_N^*(x+w)|  \frac{\d^2 w}{|w|} \
 & \le L \sum_{j=1}^n \epsilon_j^{-2} \int_{w\in \D(z_j-x, \epsilon_j)}  \frac{\d^2 w}{|w|}  + L  \int_{\substack{ |w| \le \kappa \\  (x+w) \notin \mathscr{S}_N}}   \frac{\d^2 w}{|w|} \\
 &\le 
L \sum_{j=1}^n \epsilon_j^{-2} \int_{w\in \D(z_j-z_k, \epsilon_j + \epsilon_k)}  \frac{\d^2 w}{|w|}  + 2\pi \kappa L \\
&\le L \pi \left( 1+ {\textstyle \sum_{j=1}^n} \epsilon_j^{-2} (\epsilon_j + \epsilon_k) \right) . 
\end{aligned}\]
On the other hand, as $|\widetilde{u}_N^*(z)| \le nL \epsilon^{-2}$ for all $|z| \le 3/4$, it also holds for all $x\in \D_{1/2}$, 
\[
\int_{|w| \le \kappa}  |\widetilde{u}_N^*(x+w)|  \frac{\d^2 w}{|w|} =\O(\epsilon^{-2}) . 
\]
Combining these two bounds with \eqref{error7}, we conclude that
\[ \begin{aligned}
\left| \mathfrak{Z}_N \right| 
\le
\pi L^2  \sum_{k, j =1}^n \epsilon_k^{-2} \epsilon_j^{-2}  (\epsilon_j + \epsilon_k) 
  \int_{|x-z_k| \le \epsilon_k} \hspace{-.3cm} |\dbar g_N(x)|  \d^2x 
+\O\left(\epsilon^{-2}  \int |\dbar g_N(x)| \d^2x\right)
\end{aligned}\]
By the Cauchy--Schwartz inequality and \eqref{Sigmag}, this implies that 
\[
\left| \mathfrak{Z}_N \right| 
\le   (\pi L)^2 \Sigma   \sum_{k, j =1}^n \epsilon_k^{-1} \epsilon_j^{-2}  (\epsilon_j + \epsilon_k)   + \O(\epsilon^{-2} \Sigma) . 
\]
Since our parameters $\epsilon_1, \dots , \epsilon_n \ge \epsilon$, we have  $  \sum_{k, j =1}^n \epsilon_k^{-1} \epsilon_j^{-2}  (\epsilon_j + \epsilon_k)  \le 2n^2 \epsilon^{-2}$. Hence, we have proved that 
\begin{equation} \label{error3}
\left| \mathfrak{Z}_N \right|   =  \O(\epsilon^{-2} \Sigma) . 
\end{equation}
Since $\epsilon \ge \delta \ge N^{-1/2}$, by combining the estimates \eqref{error4} and \eqref{error3} with  \eqref{error2}, this completes the proof.
\end{proof}

\subsection{Error of type $\mathfrak{T}^3_N$} \label{sect:error3}

\begin{proposition} \label{prop:error3}
We have
$\left| \mathfrak{T}^3_N(g_N) \right| = \O(N\eta)$
 as $N\to+\infty$.
\end{proposition}

\begin{proof}
First, let us observe that by Lemma~\ref{lem:06}, 
\[ \begin{aligned}
\left| \iint_{|z-x| \ge \delta} \frac{\dbar g_N(x) }{x-z}  | K_N^*(x,z) |^2 \d^2 x\d^2z
\right|  &\le \delta^{-1} \int |\dbar g_N(x)| \left( \int_{|z-x| > \delta}  | K_N^*(x,z) |^2\d^2z \right) \d^2 x \\
&\le C N \delta  \int |\dbar g_N(x)|  \vartheta_N(x) \d^2 x . 
\end{aligned}\]
Since $\| \nabla g_N\|_\infty = \O(\epsilon^{-1})$ and 
$\displaystyle  \int_{|x|\le 1/2} \hspace{-.3cm} \vartheta_N(x) \d^2 x  \le (n+1)\pi$, this shows that 
\begin{equation} \label{K7}
\left| \iint_{|z-x| \ge \delta} \frac{\dbar g_N(x) }{x-z}  | K_N^*(x,z) |^2 \d^2 x\d^2z
\right|   = \O(N\eta) . 
\end{equation}

\smallskip

Second, since $\displaystyle \left| \frac{\dbar g_N(x)- \dbar g_N(z)}{x-z}  \right| \le \| \nabla^2 g_N\|_\infty  = \O(\epsilon^{-2})$ for all $x,z\in\C$, we have 
\[ \begin{aligned}
&\left| \iint_{\substack{ |x| \le 1/2 \\  |z-x| \le \delta}} \hspace{-.1cm} \frac{\dbar g_N(x)- \dbar g_N(z)}{x-z} |K_N^{*}(x,z)|^2  \d^2z \d^2x -\iint_{\substack{ |x| \le 1/2 \\  |z-x| \le \delta}} \hspace{-.1cm} \frac{\dbar g_N(x)- \dbar g_N(z)}{x-z} |K_N^{\#}(x,z)|^2  \d^2z \d^2x \right|  \\
 &\qquad\le 2 \epsilon^{-2} \left( \iint_{\substack{ |x| \le 1/2 \\  |z-x| \le \delta}} \hspace{-.1cm}  |K^\#_N(z,x) |  \left|  K^*_N(z,x) - K^\#_N(z,x)  \right|\d^2z \d^2x 
+  \iint_{\substack{ |x| \le 1/2 \\  |z-x| \le \delta}} \hspace{-.1cm}   \left|  K^*_N(z,x) - K^\#_N(z,x)  \right|^2\d^2z \d^2x \right) . 
\end{aligned}\]
If we integrate the estimate \eqref{K4}, respectively \eqref{K5}, over the set $|x| \le 1/2$, we obtain
\[
  \iint_{\substack{ |x| \le 1/2 \\  |z-x| \le \delta}} \hspace{-.1cm}   \left|  K^*_N(z,x) - K^\#_N(z,x)  \right|^2\d^2z \d^2x   = \O(\eta^2) ,
\]
and 
\[
 \iint_{\substack{ |x| \le 1/2 \\  |z-x| \le \delta}} \hspace{-.1cm}  |K^\#_N(z,x) |  \left|  K^*_N(z,x) - K^\#_N(z,x)  \right|\d^2z   = \O(N \delta^2). 
\]
Here we used again that $\displaystyle  \int_{|x|\le 1/2} \hspace{-.3cm} \vartheta_N(x) \d^2 x  \le (n+1)\pi$.
These bounds imply that 
\begin{equation} \label{K6}
\left| \iint_{\substack{ |x| \le 1/2 \\  |z-x| \le \delta}} \hspace{-.1cm} \frac{\dbar g_N(x)- \dbar g_N(z)}{x-z} |K_N^{*}(x,z)|^2  \d^2z \d^2x -\iint_{\substack{ |x| \le 1/2 \\  |z-x| \le \delta}} \hspace{-.1cm} \frac{\dbar g_N(x)- \dbar g_N(z)}{x-z} |K_N^{\#}(x,z)|^2  \d^2z \d^2x \right|   = \O(N \eta^2) . 
\end{equation}
By symmetry, since $\supp(g_N) \subseteq \D_{1/2}$, 
\[
\left| \iint_{|z-x| \le \delta} \frac{\dbar g_N(x) }{x-z}  | K_N^*(x,z) |^2 \d^2 x\d^2z
\right|  \le  \left| \iint_{\substack{ |x| \le 1/2 \\  |z-x| \le \delta}} \hspace{-.1cm} \frac{\dbar g_N(x)- \dbar g_N(z)}{x-z} |K_N^{*}(x,z)|^2  \d^2z \d^2x \right| .
\]
Then,  using the estimate \eqref{K6} and Lemma~\ref{lem:07}, we obtain 
\begin{equation} \label{K8}
\left| \iint_{|z-x| \le \delta} \frac{\dbar g_N(x) }{x-z}  | K_N^*(x,z) |^2 \d^2 x\d^2z \right| 
= \O(N \eta^2) .
\end{equation}
Finally, it remains to combine the estimates \eqref{K7} and \eqref{K8} to complete the proof. 
\end{proof}

\subsection{Proof of Proposition~\ref{thm:exp}} \label{sect:proof}

We are now ready to give the proof of  Proposition~\ref{thm:exp}. Recall that we use the notation of Section~\ref{sect:notation}. When we combine Propositions~\ref{prop:error1}, \ref{prop:error2} and \ref{prop:error3}, we obtain that as $N\to+\infty$,
\[
\left|\mathfrak{T}^1_N(g_N) + \mathfrak{T}^2_N(g_N) - \mathfrak{T}^3_N(g_N)\right|
 = \O\left(N \eta\Sigma(g_N) (1+\eta \Sigma(g_N)) \right) ,
\]
where, by Remark~\ref{rk:uniformity}, the error term is uniform for all $\vec{z} \in \D_{\epsilon_0}^{\times n}$ , all $t\in(0,1]$ and all $\vec{\gamma} \in [-R,R]^{n}$ for a fixed $R>0$.
Since $\Sigma^2(g_N) = \O(\log N)$ according to the asymptotics \eqref{cov1} and $\eta = \delta/\epsilon = (\log N)^{\beta/2} N^{-\alpha}$, this implies that as $N\to+\infty$
\begin{equation} \label{error8}
\frac{1}{N}\left|\mathfrak{T}^1_N(g_N) + \mathfrak{T}^2_N(g_N) - \mathfrak{T}^3_N(g_N)\right|
 = \O\left( (\log N)^{\frac{\beta+1}{2}} N^{-\alpha}\right).
\end{equation}

\medskip

The main idea of the proof, which originates from \cite{Johansson98} is to observe that for any $t>0$, 
\[
\frac{d}{dt}\log \E_N\left[\exp\left( t \X(g_N) \right)\right] 
= \E_N^*\left[ \X(g_N) \right] . 
\]
Hence, by Proposition~\ref{prop:ward} applied to the function $g_N = g_N^{\g, \vec{z}}$, using the estimate \eqref{error8}, we conclude that
\begin{equation} \label{error9}
\frac{d}{dt} \log\E_N\left[\exp\left( t\X(g_N^{\g, \vec{z}}) \right)\right] 
= t \Sigma^2(g_N^{\g, \vec{z}})+  \O\left( (\log N)^{\frac{\beta+1}{2}} N^{-\alpha}\right) ,
\end{equation}
where the error term is uniform for all $t\in[0,1]$, all $\vec{\gamma}$ in compact subsets of $\R^n$ and all $\vec{z} \in \D_{\epsilon_0}^{\times n}$. Then, if we integrate the asymptotics \eqref{error9} for $t\in [0,1]$,  we 
obtain \eqref{exp}.

\section{Kernel asymptotics} \label{sect:approx}

In this section, we obtain the asymptotics for the correlation kernel induced by the  biased measure \eqref{Pbias} that we need in Section~\ref{sect:clt} in order to control the error term in Ward's equation. 
Let us introduce
\begin{equation} \label{normQ}
\| f\|_Q^2 = \int_\C |f(x)|^2 e^{-2N Q(x)} \d^2x , 
\end{equation}
and similarly for the norm $\|\cdot\|_{Q^*}$. 
Recall that $Q(x) = |x|^2/2$ is the Ginibre potential and $Q^* = Q- \frac{g_N}{2N}$ is a potential which is perturbed by the function $g_N = g_N^{\g, \vec{z}} \in \Co_c^\infty(\D_{1/2})$  given by \eqref{g} with $\vec{z} \in \D_{\epsilon_0}^{\times n}$ and $\gamma\in[-R,R]^n$ for some fixed $n\in \N$ and $R>0$.
We rely on the Conventions~\ref{conv} and we choose 
$N_0 \in \N$ sufficiently large so that $\eta\le 1/4$ and  $\| \Delta g_N\|_\infty \le N$ for all $N\ge N_0$. 

%Let us also recall that, like in Section~\ref{sect:notation}, we have the parameters:
%\[
%\epsilon=\epsilon(N) = N^{-1/2+\alpha}, \qquad
%\delta = \sqrt{(\log N)^\beta /N} 
%\qquad\text{and}\qquad
%\eta= \eta(N) = \delta/\epsilon = (\log N)^{\beta/2} N^{-\alpha} , 
%\]
%where $0<\alpha<1/2$ and $\beta>1$. Moreover, we have chosen $\ell\in\N$ in such a way that $\eta^\ell \le N^{-1}$. 

\subsection{Uniform estimates for the  1--point function}

In this section, we collect some simple estimates on the 1--point function $u_N^*$ which we need. We skip the details since the argument is the same as in \cite[Section~3]{AHM10} only adapted to our situation.

\begin{lemma} \label{lem:01}
There exists a universal constant $C>0$ such that if $N\ge N_0$,  for any function $f$ which is analytic in $\D(z; 2/\sqrt{N})$ for some $z\in\C$, 
\[
|f(z)|^2 e^{-2N Q^*(z)} \le CN \|f\|^2_{Q^*} . 
 \]
\end{lemma}

\begin{proof}
If $N\ge N_0$, we have $\Delta Q^* \le3$ and by \cite[Lemma~3.1]{AHM10}, we obtain 
\[
|f(z)|^2 e^{-2N Q^*(z)} \le  N \int_{|z-x| \le N^{-1/2}} |f(x)|^2 e^{-2 N Q^*(x)} e^{3N|x-z|^2} \sigma(\d x) .
\]
This immediately yields the claim. 
\end{proof}

\begin{lemma} \label{lem:02}
With the same $C>0$ as in Lemma~$\ref{lem:01}$,  it holds  for all  $N\ge N_0$ and all $z\in\C$,
\[
u_N^*(z)  \le C N . 
\]
\end{lemma}

\begin{proof}
Fix $z\in\C$ and let us apply Lemma~\ref{lem:01} to the polynomial $k^*_N(\cdot,z)$, we obtain
\[
|k^*_N(w,z)|^2 e^{-2 N Q^*(w)}  \le  C N k^*_N(z,z) ,
\]
since $\| k^*_N(\cdot,z)\|_{Q^*}^2 = k^*_N(z,z)$ because of the reproducing property of the kernel $k^*_N$. Taking $w=z$ in the previous bound and using that 
$u_N^*(z) =k^*_N(z,z) e^{-2N Q^*(z)}$, we obtain the claim. 
\end{proof}

\subsection{Preliminary Lemmas}

Recall that we let $\Upsilon_N^w(u) =  \sum_{i=0}^{\ell} \frac{u^i}{i!} \partial^i g_N(w) $ and that we defined the \emph{approximate Bergman kernel}  $k^\#$ by 
\begin{equation*}
k^\#_N(x,w) = \frac{N}{\pi} e^{N x \overline{w}}  e^{- \Upsilon_N^w(x-w) } ,
\qquad x,w\in\C. 
\end{equation*}
We note that this kernel is not Hermitian but it is analytic in $x\in\C$  and we define the corresponding operator:
\begin{equation} \label{Kapprox} 
K^\#_N[f] (w) = \int_\C  \overline{k^\#_N(x,w)} f(x) e^{-2N Q^*(x)} \d^2x  , \qquad w\in \C, 
\end{equation}
for any $f\in L^2(e^{-2N Q^*})$. 
According to \eqref{kernel1}, we make a similar definition for $K^*_N[f] $. 
Our next Lemma is the counterpart of \cite[Lemma A.2]{AHM15} and it relies on the analytic properties of the function $\Upsilon_N^w$. Since the test function $g_N$ develops logarithmic singularities for large $N$, we need to adapt the proof accordingly. 

\begin{assumption}  \label{ass:chi}
Let $\chi \in \Co^\infty_c\big(\D_{2\delta}\big)$  be a radial function such that  $0\le \chi\le 1$, $\chi =1$ on $\D_{\delta}$, and $\| \nabla \chi\|_\infty \le C \delta^{-1}$ for a $C>0$ independent of $N\in\N$.  
In the following for any $z\in\C$, we let $\chi_z = \chi(\cdot-z)$.
\end{assumption}
 
\begin{lemma} \label{lem:1}
There exists a constant $C>0$ (which depends only on $R>0$, the mollifier $\phi$ and $n, \ell\in\N$) 
such that for any $z\in\C$  and any function $f \in  L^2(e^{-2N Q^*})$ which is analytic in $\D(z, 2 \delta)$, 
\[
\left| f(z) -K^\#_N[\chi_z f] (z) \right| \le C \big(  N^{-1/2} \vartheta_N(z) + e^{- N \delta^2/2}  \big)  e^{NQ^*(z)}  \| f\|_{Q^*} ,
\]
where $\vartheta_N$ is as in Proposition~\ref{prop:approx}.  
\end{lemma}

\begin{proof}
We fix $z\in\C$ and by definitions, 
\[ \begin{aligned}
K^\#_N[\chi_z f] (z)  & = \frac{N}{\pi}  \int   e^{ g_N(x)- \overline{\Upsilon_N^z(x-z)} }
\chi_z(x) f(x) e^{N(z-x) \overline{x} } \d^2x \\
& =   \int   e^{ g_N(x)- \overline{\Upsilon_N^z(x-z)} }
 \chi_z(x) f(x)  \dbar \left( e^{N(z-x) \overline{x} } \right) \frac{1}{z-x} \frac{\d^2x}{\pi} . 
\end{aligned}\]
By formula \eqref{cauchy}, since $\chi_z(z) =1$ and $\Upsilon_N^z(0) =g_N(z) \in\R$, we obtain
\[
K^\#_N[\chi_z f] (z)
= f(z)  - \int   \dbar \left(  e^{ g_N(x)- \overline{\Upsilon_N^z(x-z)} }\chi_z(x) f(x) \right)  e^{N(z-x) \overline{x} }  \frac{1}{z-x} \frac{\d^2x}{\pi} .
\]
Since  $f$ is analytic in $\D(z,2 \delta) = \supp(\chi_z)$, this implies that 
\begin{equation} \label{est8}
\begin{aligned} 
f(z) - K^\#_N[\chi_z f] (z)
= \mathfrak{Z}_N
+ \int   e^{ g_N(x)- \overline{\Upsilon_N^z(x-z)}} f(x) \dbar\chi_z(x)   \frac{ e^{Nz\overline{x} } }{z-x} e^{-2NQ(x)} \frac{\d^2x}{\pi} ,
\end{aligned}
\end{equation}
where we let
\begin{equation} \label{est10}
 \begin{aligned}
\mathfrak{Z}_N  &: = \int   \dbar \left(  e^{ g_N(x)-\overline{\Upsilon_N^z(x-z)}} \right)  \chi_z(x) f(x)   \frac{e^{N(z-x) \overline{x} } }{z-x} \frac{\d^2x}{\pi}   \\
&=  - \int_{ |x-z| \le 2\delta}  \hspace{-.3cm} \frac{\overline{ \partial  g_N(x) - \partial \Upsilon_N^z(x-z)}}{x-z} e^{ g_N(x)- \overline{\Upsilon_N^z(x-z)}}   \chi_z(x) f(x) e^{Nz\overline{x} } e^{-2NQ(x)} \frac{\d^2x}{\pi}  . 
\end{aligned}
\end{equation}
Using the Assumptions~\ref{ass:chi}, the second term on the RHS of  \eqref{est8} satisfies 
\begin{equation} \label{est7}
\left|\int   e^{ g_N(x)- \overline{\Upsilon_N^z(x-z)} } f(x) \dbar\chi_z(x) \frac{ e^{Nz\overline{x} } }{z-x} e^{-2NQ(x)}\frac{\d^2x}{\pi} \right| 
\le  \delta^{-2} \int_{\delta \le |x-z| \le 2\delta} \hspace{-.3cm} |f(x)| \  |e^{  g_N(x)-\Upsilon_N^z(x-z) }|\ e^{N\Re\{ z\overline{x} \}} e^{-2NQ(x)} \frac{\d^2x}{\pi}  . 
\end{equation}
Recall $\| \nabla^k g_N\|_\infty \le C \epsilon^{-k}$ for $k=1, \dots, \ell$ and we assume that $\eta= \delta/\epsilon  \le1/4$. Then, by Taylor's formula, if $|x-z| \le 2\delta$, 
 \begin{equation} \label{est11}
 e^{  g_N(x)-\Upsilon_N^z(x-z)} = e^{ g_N(x)/2 -  g_N(z)/2} e^{- \i\Im\{ \partial  g_N(z) (x-z) \}   + \O(\eta^2) } . 
 \end{equation}
This shows that  on the RHS of \eqref{est7}, 
 $|e^{  g_N(x)-\Upsilon_N^z(x-z) }| \le C e^{ g_N(x)/2 -  g_N(z)/2} $. 
 Moreover, by  rearranging \eqref{Q}, 
 \begin{equation}  \label{est15}
  \frac{g_N(x) -  g_N(z)}{2} - N \big(2Q(x) - \Re\{z\overline{x}\}  \big)=  -\frac N2 |z-x|^2 - NQ^*(z) - NQ^*(x), 
\end{equation}
which shows that
\[
\left|\int   e^{ g_N(x)-\overline{\Upsilon_N^z(x-z)}} f(x) \dbar\chi_z(x)  \frac{ e^{Nz\overline{x} } }{z-x} e^{-2NQ(x)} \frac{\d^2x}{\pi} \right|
 \le C \delta^{-2} e^{NQ^*(z)}
\int_{|x-z| \ge \delta} \hspace{-.3cm} |f(x)|  e^{-N |x-z|^2/2}  e^{-NQ^*(x)} \frac{\d^2x}{\pi} . 
\] 
By Cauchy--Schwartz inequality and \eqref{normQ}, we conclude that the second term on the RHS of \eqref{est8} is bounded by 
\begin{equation} \label{est9}
\left|\int   e^{ g_N(x)- \overline{\Upsilon_N^z(x-z)} } f(x) \dbar\chi_z(x) \frac{ e^{Nz\overline{x} } }{z-x} e^{-2NQ(x)} \frac{\d^2x}{\pi} \right| \le C  e^{NQ^*(z)}
\| f\|_{Q^*}  e^{- N \delta^2/2}  . 
\end{equation}
Here we used that $\delta^{-2} \le N$ and that for any $r>0$, 
\begin{equation} \label{est16}
   \int_{ |x-z| \ge r} \hspace{-.3cm}    
 e^{-N |x-z|^2 } \frac{\d^2x}{\pi} = N^{-1} e^{-N r^2}  . 
 \end{equation}

The RHS of \eqref{est9} will be negligible and it remains to control $\mathfrak{Z}_N$. 
Using again formulae \eqref{est11}--\eqref{est15} and taking the $|\cdot|$ inside the integral \eqref{est10}, we obtain
 \begin{equation} \label{est17}
|\mathfrak{Z}_N  |
\le   C e^{NQ^*(z)} \int_{ |x-z| \le 2\delta}  \bigg|  \frac{ \partial  g_N(x) - \partial \Upsilon_N^z(x-z)}{\overline{x-z}}  \bigg|   | f(x) |e^{- N|z-x|^2/2 } e^{-NQ^*(x)} \frac{\d^2x}{\pi}  , 
\end{equation}
where we used that $\| \chi_z\|_\infty \le 1$. 
Since the function $g_N$ is smooth, by Taylor's Theorem up to order $2\ell$, it holds for any $|u| \le 2\delta$, 
\[ \begin{aligned}
\partial  g_N(z+u) - \partial \Upsilon_N^z(u) 
&=   {\textstyle  \sum_{i=0}^{\ell} }  {\textstyle  \sum_{j=1}^{\ell} }  \M_{i,j} u^i \overline{u}^j \
\partial^{i+1} \dbar^j  g_N(z) 
%\frac 14 {\textstyle  \sum_{i=1}^{\ell-1} } \partial^{i-1} \Delta g_N(\zeta_i) \overline{(x-z)} (x-z)^{i-1}
+ \O\left( |u| \sup_{\ell \le k \le 2\ell}  \big\{  \| \nabla^{k+2}  g_N\|_\infty \delta^{k}\big\} \right) ,
\end{aligned}\]
where the coefficients $  \M_{i,j} >0$. 
Let us recall that $\| \nabla^{k}  g_N\|_\infty \le C \epsilon^{-k}$  for any integer $k\in[1,2\ell]$, $\eta =\delta/\epsilon$ is small and we fixed $\ell \in\N$ in such a way that the parameter $\eta^\ell \le N^{-1}$. 
In particular, we have constructed $\Upsilon_N^z$ in such a way that we deduce from the previous expansion that for any  $x\in \D(z,2\delta)$, 
\[
\partial  g_N(x) - \partial \Upsilon_N^z(u) 
= \frac{\overline{u}}{4}  {\textstyle  \sum_{i=0}^{\ell} }  {\textstyle  \sum_{j=0}^{\ell-1} }   \M_{i,j+1} u^i \overline{u}^{j} \
\partial^{i} \dbar^j (\Delta g_N)(z)  +\O(|u|N^{-2\alpha}) , \qquad \text{where } u = x-z 
\]
and we have used that $\partial \dbar g_N = \frac 14 \Delta g_N$. 
Using \eqref{g}, \eqref{psi} and the definition of $\vartheta_N$, it is straightforward to verify that for any integer $k\in [0,2\ell]$ and uniformly for all $z\in\C$,  
$| \nabla^k (\Delta g_N)(z)| \le C \epsilon^{-k} \vartheta_N(z) $. 
Hence, these estimates imply that uniformly  for all  $x\in \D(z,2\delta)$, 
\begin{equation} \label{est13}
\left| \frac{ \partial  g_N(x) - \partial \Upsilon_N^z(x-z)}{\overline{x-z}} \right| \le 
C  \vartheta_N(z)  +  \O\left(N^{-2\alpha}\right) . 
\end{equation}
Note that the error term is 0 if $z\notin \D$ since $g_N$ has compact support in $\D_{\epsilon_0}$. 
Therefore, by combining \eqref {est17} and \eqref{est13}, we conclude that   \begin{equation} \label{est12} 
|\mathfrak{Z}_N  |
\le   C N^{-1/2} \vartheta_N(z) e^{NQ^*(z)}  \| f\|_{Q^*} ,
 \end{equation}
where we have used the Cauchy--Schwartz inequality and \eqref{est16} with $r=0$. 
Finally, by combining the estimates \eqref{est9} and \eqref{est12}  with formula \eqref{est8}, this completes the proof.
\end{proof}

Our next Lemma is the counterpart of \cite[(4.12)]{AHM15}. The proof needs again to be carefully adapted but the general strategy remains the same as in \cite{AHM15} and relies on H\"ormander's inequality and the fact that  \eqref{kernel1} is the reproducing kernel of the Hilbert space $\mathscr{P}_N \cap L^2(e^{-2NQ^*})$.

\begin{lemma} \label{lem:2}
For any integer $\kappa\ge0$,  there exists $N_\kappa\in\N$  (which depends only on $R>0$, the mollifier $\phi$ and $n, \ell\in\N$)  such that if $N \ge N_\kappa$, we have  for all $z\in \D_{1-2\delta}$ and all $w\in\D(z,\delta)$, 
\[
\left| k^\#_N(w,z) -  K^*_N[\chi_z k^\#_N(\cdot,z)] (w)\right| \le N^{-\kappa}  e^{NQ^*(z) + N Q^*(w)} . 
\]
\end{lemma}

\begin{proof}
In this proof, we fix $N$ and  $z\in \D_{1-2\delta}$. We let $f := \chi_z k^\#_N(\cdot,z)$ where $\chi_z$ is as in Assumptions~\ref{ass:chi} and $W(x) := N \big(\varphi(x) +1/2\big) + \log\sqrt{1+|x|^2}$ where $\varphi$ is as in equations \eqref{varphi}--\eqref{EL}. 
%Note that by construction $\supp(f) \subset \overline{\D}$ and $f$ is analytic in $\D(z,\delta)$. 
Let also $V$ be the minimal solution   in $L^2(e^{-2W})$ of the problem $\dbar V= \dbar f$ and recall that H\"ormander's inequality, e.g. \cite[formula (4.5)]{AHM10}, for the  $\dbar$--equation states that
\[
\|V\|_{L^2(e^{-2W})}^2 \le 2 \int_\D \left| \dbar f(x)\right|^2 \frac{e^{-2W(x)}}{\Delta W(x)} \d^2x .
\]
 Here we used that $W$ is strictly subharmonic\footnote{Note that we have $\Delta \left( \log\sqrt{1+|x|^2} \right) = \frac{4}{(1+|x|^2)^2} >0$ for $x\in\C$.}. By \eqref{EL}, since $W(x) \ge N Q(x) $  and $\Delta W(x) \ge N \Delta Q(x) = 2N$ for all $x\in\D$, this implies that 
\[
\|V\|_{L^2(e^{-2W})}^2 \le N^{-1} \|\dbar f\|_{Q}^2  . 
\]
Moreover, by \eqref{varphi}, there exists a universal constant $c>0$ such that $W(x) \le N Q(x) + c $.
Therefore, we obtain
\begin{equation} \label{est3}
\|V\|_{Q}^2 \le e^{2c} N^{-1} \|\dbar f\|_{Q}^2  . 
\end{equation}

Recall that $Q^*= Q -\frac{g_N}{2N}$ where the perturbation $g_N$ is given by \eqref{g} and  satisfies  $\| g_N\|_\infty \le C \log \epsilon^{-1}$.
% for some constant $C_{\boldsymbol{\gamma}}>0$ independent of $N\in\N$. 
This implies that $L^2(e^{-2NQ^*}) \cong L^2(e^{-2NQ})$ with for any function $h\in L^2(e^{-2NQ})$:
\[
\epsilon^{C/2}\|h\|_{Q^*} \le \| h\|_{Q} \le \epsilon^{-C/2} \|h\|_{Q^*} . 
\]
By \eqref{est3}, this equivalence of norms shows that if $N\in\N$ is sufficiently large, 
\begin{equation} \label{est14}
\|V\|_{Q^*}^2 \le N^{C-1} \|\dbar f\|_{Q^*}^2  . 
\end{equation}

Observe that by \eqref{varphi}, $W(x) = (N+1) \log|x| +o(1)$ as $|x| \to +\infty$, hence  the Bergman space $A^2(e^{-2W})$ coincide with $\mathscr{P}_N$ and we must have $V- f \in \mathscr{P}_N$, see \eqref{polynom}. 

\medskip 

Now, we let $U$ to be the minimal solution in $L^2(e^{-2NQ^*})$ of the problem $U-f \in \mathscr{P}_N$. 
Since $U$ has minimal norm, \eqref{est14}  implies that 
\begin{equation} \label{est4}
 \|U\|_{Q^*}^2 \le N^{C-1} \|\dbar f\|_{Q^*}^2 . 
\end{equation}

Since $k^\#_N(\cdot,z)$ is analytic,  see \eqref{kapprox}, 
$\dbar f =  k^\#_N(\cdot,z) \dbar \chi_z$ and  according to the Assumptions~\ref{ass:chi},  
\[
 \|\dbar f\|_{Q^*}^2  \le C \delta^{-2} \int_{\delta \le |x-z| \le 2\delta} \hspace{-.3cm}|  k^\#_N(x,z) |^2 e^{-2NQ^*(x)} \d^2x . 
\]
Recall that $\eta = \delta/\epsilon \le 1/4$ and $\| \nabla^i g_N\| \le C \epsilon^{-i} $ for all $i=1,\dots, \ell$. 
Then by \eqref{kapprox} with $t=1$, it holds for all $z \in \C$ and $|x-z| \le 2\delta$, 
$\Upsilon_N^z(x-z) = g_N(z) + \O(\eta)$ so that
\[
| k^\#_N(x,z) |^2 \le C e^{-2 g_N(z)+ 2N  \Re\{ x \overline{z} \}} .
\]
By \eqref{est15} and using that $|g_N(x)-g_N(z)| =\O(\eta)$,  this shows that
\[
| k^\#_N(x,w) |^2 e^{-2NQ^*(x)} \le C e^{ 2N Q^*(z) - N |x-z|^2} .
\]
Then by \eqref{est16} and using that $\delta^{-2} \le N$, we obtain
\[ \begin{aligned}
 \|\dbar f\|_{Q}^2  & \le C \delta^{-2} e^{2NQ^*(z)}\int_{\delta \le |x-z| \le 2\delta}  \hspace{-.3cm} e^{- N |x-z|^2} \frac{\d^2x}{\pi} \\
 & \le   C e^{2NQ^*(z)} e^{-N \delta^2} . 
\end{aligned}\]
Combining the previous estimate with \eqref{est4}, we conclude that
\begin{equation} \label{est5}
\|U\|_{Q^*}^2 \le C  N^{C-1} e^{2NQ^*(z)} e^{-N \delta^2}. 
\end{equation}
We may now  turn \eqref{est5} into a pointwise estimate using Lemma~\ref{lem:01}. 
Note that both $f$ and $U$ are analytic\footnote{Here we used that $\chi_z=1$ on $\D(z;\delta)$ so that $\dbar f=0$ and that $\dbar f = \dbar U$ since $U-f \in\mathscr{P}_N$ by definition of $U$. } in $\D(z;\delta)$, this implies that for any $w\in \D(z;\delta)$
\begin{equation} \label{est6}
|U(w)|^2 e^{-2N Q^*(w)} \le C  N^{C}  e^{2NQ^*(z)}   e^{-N \delta^2}. 
\end{equation}
Since $k^*_N$ is the reproducing kernel of the Hilbert space $\mathscr{P}_N \cap L^2(e^{-2NQ^*})$, see  \eqref{kernel1}, it is well known that minimal solution $U$ is given by
\begin{equation*} 
U = f - K^*_N[f] . 
\end{equation*}
Conseqently, as  $f = \chi_z k^\#_N(\cdot,z)$ and $\chi_z =1$ on $\D(z;\delta)$, by \eqref{est6},  we conclude that for any $w\in \D(z;\delta)$, 
\[
\left| k^\#_N(w,z) -  K^*_N[\chi_z k^\#_N(\cdot,z)] (w) \right| \le C  N^{C/4} e^{NQ^*(z) + NQ^*(w)} e^{- N \delta^2/2} . 
\]
Since $e^{N \delta^2}$ grows faster than any power of $N\in\N$, this completes the proof. 
\end{proof}

We are now ready to give the proof of our main approximation for the correlation kernel $K^*_N$, see \eqref{kernel2}.

\subsection{Proof of Proposition~\ref{prop:approx}} \label{sect:BA}
%\begin{proof}[Proof of Proposition~\ref{prop:approx}]

We apply Lemma~\ref{lem:1}  to the function $f(x) = k^*_N(x,w)$ which is analytic for $x\in\C$ with norm
\[
\| f \|_{Q^*}^2 =  k^*_N(w,w) = u^*_N(w) e^{2N Q^*(w)} , 
\]
by the reproducing property. Hence, we obtain for any  $z\in\D$ and $w\in\C$, \begin{equation*}
\left|  k^*_N(z,w) -K^\#_N[\chi_z k^*_N(\cdot,w)] (z) \right| \le C \vartheta_N(z)\sqrt{\tfrac{u^*_N(w)}{N}} e^{NQ^*(z) + N Q^*(w)}    . 
\end{equation*}
By Lemma~\ref{lem:02}, this shows that 
\begin{equation} \label{est1}
\left|  k^*_N(z,w) -K^\#_N[\chi_z k^*_N(\cdot,z)]  (w) \right| 
\le  C\vartheta_N(z)  e^{NQ^*(z) + N Q^*(w)}    . 
\end{equation}
Recall that by \eqref{Kapprox}, we have
\[
K^\#_N[\chi_z k^*_N(\cdot,w)] (z) =  \int  \overline{k^\#_N(x,z)} \chi_z(x) k^*_N(x,w) e^{-2N Q^*(x)} \d^2x ,
\]
so that
\[
\overline{K^\#_N[\chi_z k^*_N(\cdot,w)] (z)} =  \int k^\#_N(x,z) \chi_z(x)  \overline{k^*_N(x,w)} e^{-2N Q^*(x)} \d^2x
= K^*_N[\chi_z k^\#_N(\cdot,z)]  (w) .
\]
Then, since the kernel $k^*_N$ is Hermitian,   it follows from the bound \eqref{est1} that  for any  $z, w\in\D$, 
\begin{equation} \label{est2}
\left|  k^*_N(w,z) -K^*_N[\chi_z k^\#_N(\cdot,z)]  (w) \right| 
\le  C\vartheta_N(z)  e^{NQ^*(z) + N Q^*(w)}  . 
\end{equation}
Finally, by Lemma~\ref{lem:2}, we conclude that for any $z\in \D_{1-2\delta}$ and all $w\in\D(z,\delta)$, 
\[
\left|k^*_N(w,z) - k^\#_N(w,z) \right| \le   C\vartheta_N(z)  e^{NQ^*(z) + N Q^*(w)} .
\]
\end{document}